\documentclass[reqno,twoside]{amsart}
\usepackage{amsfonts}
\usepackage{amsmath,amssymb}
\usepackage[dvips,bottom=1.4in,right=1in,top=1in, left=1in]{geometry}
\usepackage{epic}
\usepackage{pstricks}
\usepackage{graphicx}
\usepackage{xcolor}

\usepackage{eucal}

\usepackage{graphicx}

\newtheorem{theorem}{Theorem}[section]
\newtheorem{proposition}[theorem]{Proposition}
\newtheorem{remark}[theorem]{Remark}
\newtheorem{lemma}[theorem]{Lemma}
 
\newtheorem{definition}[theorem]{Definition}

\numberwithin{equation}{section}

\graphicspath{{./Imgs/}}

\newcommand{\be}{\begin{equation}}
\newcommand{\ee}{\end{equation}}
\newcommand{\ba}{\begin{eqnarray}}
\newcommand{\ea}{\end{eqnarray}}
\newcommand{\beq}{\begin{equation}}
\newcommand{\eeq}{\end{equation}}

\numberwithin{equation}{section}

\def\Omc{\mathbb{R}^N\setminus\Omega}

\def\Omb{\mathbb{R}^N\setminus\overline{\Omega}}

\def\RR{{\mathbb{R}}}
\def\NN{{\mathbb{N}}}

\def\CC{{\mathbb{C}}}

\def\Om{\Omega}

\def\bOm{\overline{\Om}}

\def\pOm{\partial \Omega}

\usepackage{mathtools}

\mathtoolsset{showonlyrefs}

\usepackage{color}
\usepackage{stmaryrd}

\newcommand{\D}{\displaystyle}



\begin{document}

\title[Controllability properties of the nonlocal MGT equation]{Exterior controllability properties of a nonlocal Moore--Gibson--Thompson equation}

\author{Carlos Lizama}
\address{C. Lizama, Universidad de Santiago de Chile, Facultad de Ciencia, Departamento de
Matem\'atica y Ciencia de la Computaci\'on, Casilla 307-Correo 2,
Santiago, Chile.}
 \email{carlos.lizama@usach.cl}

\author{Mahamadi Warma}
\address{M. Warma, University of Puerto Rico, Rio Piedras Campus, Department of Mathematics,
 Faculty of Natural Sciences,  17 University AVE. STE 1701  San Juan PR 00925-2537 (USA)}
\email{mahamadi.warma1@upr.edu, mjwarma@gmail.com }

\author{Sebasti\'an Zamorano}
\address{S. Zamorano, Universidad de Santiago de Chile, Facultad de Ciencia, Departamento de
Matem\'atica y Ciencia de la Computaci\'on, Casilla 307-Correo 2,
Santiago, Chile.}
 \email{sebastian.zamorano@usach.cl}

\thanks{C. Lizama is partially supported by Fondecyt Grant NO. 1180041. The work of M. Warma is partially supported by the Air Force Office of Scientific Research (AFOSR) under Award NO:  FA9550-18-1-0242. S. Zamorano is supported by the Fondecyt Postdoctoral Grant NO. 3180322.}

\begin{abstract}
The three concepts of \emph{exact, null and approximate} controllabilities  are analyzed  from the exterior of the Moore--Gibson--Thompson equation asso\-cia\-ted with the fractional Laplace operator subject to the nonhomogeneous Dirichlet type exterior condition.  Assuming that $b>0$ and $\alpha-\frac{c^2}{b}>0$, we show that if $0<s<1$ and $\Omega\subset\RR^N$ ($N\ge 1$) is a bounded Lipschitz domain, then there is no control function $g$ such that the  following system 
\begin{align*}
\begin{cases}
u_{ttt} + \alpha u_{tt}+c^2(-\Delta)^{s} u + b(-\Delta)^{s} u_{t}=0 & \mbox{ in }\; \Omega\times(0,T),\\
u=g\chi_{\mathcal O\times (0,T)} &\mbox{ in }\; (\Omc)\times (0,T) ,\\
u(\cdot,0) = u_0, u_t(\cdot,0) = u_1, u_{tt}(\cdot,0)=u_2  &\mbox{ in }\; \Omega,
\end{cases}
\end{align*}
is exact or null controllable at time $T>0$. However, we prove that for every  $0<s<1$, the system is indeed
approximately controllable for any $T>0$ and $g\in \mathcal D(\mathcal O\times(0,T))$, where $\mathcal O\subset\Omc$ is an arbitrary non-empty open set.  
\end{abstract}

\date{}
\maketitle

\section{Introduction}\label{intro}

In the present work we investigate  the following third order nonlocal partial differential equation

\begin{align}\label{SD-WE}
\begin{cases}
u_{ttt} + \alpha u_{tt}+c^2(-\Delta)^{s} u + b(-\Delta)^{s} u_{t}=0 & \mbox{ in }\; \Omega\times(0,T),\\
u=g\chi_{\mathcal O\times (0,T)} &\mbox{ in }\; (\Omc)\times (0,T) ,\\
u(\cdot,0) = u_0, u_t(\cdot,0) = u_1, u_{tt}(\cdot,0)=u_2  &\mbox{ in }\; \Omega,
\end{cases}
\end{align}
that we call a linearized {\it nonlocal} version of the so called Moore-Gibson-Thompson (MGT) equation  \cite{Las,LW,LW2}. In \eqref{SD-WE}, $\Omega\subset\RR^N$ is a bounded open set with a Lipschitz continuous boundary $\pOm$, $\alpha$, $b$, $c$ are real numbers, $(-\Delta)^s$ ($0<s<1$) is the fractional Laplace operator (see \eqref{fl_def}), $u=u(x,t)$ is the state to be control and $g=g(x,t)$ is the control function which is localized in an open set $\mathcal O\subset \Omc$. 

Despite the wide range of applications of the local MGT equation such as the medical and industrial use of high intensity ultrasound in lithotripsy, thermotherapy, ultra-sound cleaning, etc., there have been quite a few works about their controllability properties \cite{LiZa2018}. 

For the notion controllability of PDEs from the exterior of the domain where the PDE is solved, the nonlocal case seems to be more suitable to handle because, on the one hand, the associated stationary (time independent) system is ill-posed (see e.g. \cite{War-ACE}) if the control function $g$ is prescribed at the boundary $\pOm$ and, on the other hand, it has been very recently  shown by Warma \cite{War-ACE} that for nonlocal PDEs associated with the fractional Laplacian, the exterior control  is the right notion that replaces the classical boundary control problems (that is, when the control function is localized on a subset $\omega$ of the boundary $\pOm$) associated with local operators such as the Laplace operator or general second order elliptic operators.

On the other hand, the MGT equation (which is originally a nonlinear equation) arises from modeling high amplitude sound waves. The classical nonlinear acoustics models include Kuznetsov's equation, the Westervelt equ\-ation and the Kokhlov - Zabolotskaya - Kuznetsov equation. A thorough study of the linearized models is a good starting point for better understanding the well-posedness and asymptotic behaviors of the nonlinear models. We refer to  \cite{CLD,DLP,KL,KLM,KLP,Las,LW,marchand2012abstract} and the references therein for the derivation of the local version of the MGT equation and the physical meaning of the parameters $\alpha$, $c$ and $b$.  A complete analysis concerning well-posedness, regularity, stability and asymptotic behavior of solutions has been established in the above mentioned references. However, due to the nature of the applications, it is desirable to know how the dynamics of the model changes by means of external controls or forces.


We shall show that if $b>0$ and $(u_0,u_1,u_2)$ belongs to a suitable Banach space, then for every function $g\in L^2((0,T); W^{s,2}(\Omc))$, the system \eqref{SD-WE} has a unique weak solution $(u,u_t,u_{tt})$ satisfying the regularity $u\in C([0,T];L^2(\Omega))\cap C^2([0,T];W^{-s,2}(\bOm))$.  
 In such case, the set of reachable states can be defined as follows:

\begin{align*}
\mathcal R((u_0,u_1,u_2),T):=\Big\{(u(\cdot,T),u_t(\cdot,T),u_{tt}(\cdot,T)):\; g\in L^2((0,T); W^{s,2}(\Omc))\Big\}.
\end{align*}
The classical three notions of controllability for this system can then be defined as follows.

\begin{itemize}
\item  We shall say that the system \eqref{SD-WE} is null controllable at time $T>0$ if
\begin{align*}
(0,0,0)\in \mathcal R((u_0,u_1,u_2),T).
\end{align*}
\item The system will be said to be exact controllable at $T>0$ if
\begin{align*}
\mathcal R((u_0,u_1,u_2),T)=L^2(\Omega)\times W^{-s,2}(\bOm)\times W^{-s,2}(\bOm).
\end{align*}
\item Finally we will say that the system is approximately controllable at $T>0$ if
\begin{align*}
\mathcal R((u_0,u_1,u_2),T)  \mbox{  is dense in }L^2(\Omega)\times W^{-s,2}(\bOm)\times W^{-s,2}(\bOm).
\end{align*}
\end{itemize}

From the above definitions, it is easy to see that null or exact controllability implies approximate controllability.
We refer to Section \ref{preli} for the definition of the spaces involved.

Our first main result sates that if $b>0$ and $\alpha-\frac{c^2}{b}>0$, then the system \eqref{SD-WE} is not exact or null controllable at time $T>0$. As a substitute, we obtain that the system is indeed approximately controllable at any $T>0$ and for every $g\in\mathcal D(\mathcal O\times (0,T))$ where $\mathcal O$ is an arbitrary non-empty open subset of $\Omc$.  This is the best possible result that can be obtained regarding the controllability of the system \eqref{SD-WE}. 

We remark that in our study of controllability, we shall assume that $b>0$ because if $b=0,$ then the system \eqref{Ev-Sy}  is ill-posed \cite{KLM}  and, if $b<0,$ then we lost the good behavior of the eigenvalues of the fractional Laplacian operator.

As far as we know, the present work is the first one that analyses the controllability properties for the nonlocal MGT equation using an exterior control function $g$. We can mention that our recently paper \cite{LiZa2018} is the first work dealing with the interior controllability issues for the local MGT equation using the concept of moving control, but associated with the Laplace operator.

When $g=0$, letting $(-\Delta)_D^s$ be the realization of $(-\Delta)^s$ in $L^2(\Omega)$ with the zero exterior condition $u=0$ in $\Omc$, then the associated system can be written as the following evolution system:
\begin{align}\label{Ev-Sy}
\begin{cases}
u_{ttt} + \alpha u_{tt}+c^2(-\Delta)_D^{s} u + b(-\Delta)_D^{s} u_{t}=0 & \mbox{ in }\; \Omega\times(0,T),\\
\\
u(\cdot,0) = u_0, u_t(\cdot,0) = u_1, u_{tt}(\cdot,0)=u_2  &\mbox{ in }\; \Omega.
\end{cases}
\end{align}
The well-posedness of  an abstract version of \eqref{Ev-Sy} with $(-\Delta)_D^s$ replaced with a generic self-adjoint operator $A$ with domain $D(A)$ dense in a Hilbert space $H$ has been completely examined in \cite{CLD,DLP,KL,KLM,KLP,Las,LW,LW2,marchand2012abstract} and their references by using semigroup methods. Some nonlinear models and some versions including memory terms have been also intensively studied by Lasiecka and Wang \cite{Las,LW,LW2}  where they have obtained some fundamental and beautiful results.

We notice that $(-\Delta)_D^s$ is a self-adjoint operator in $L^2(\Omega)$ with dense domain and has a compact resolvent (see Section \ref{preli}), hence it enters in the framework of the above mentioned references. However, in \eqref{SD-WE} we have a non zero exterior condition which did not satisfy the conditions contained in the above references. For this reason, in the present article we shall also include new results of existence and regularity of solutions to our nonhomogeneous system.

 
In contrast,  we observe that with the non zero exterior condition $u=g$ in $\Omc,$ the associated operator $(-\Delta)^s$ is not longer a generator of a $C_0$-semigroup and hence semigroup methods  cannot be used directly to prove the well-posedness of the system \eqref{SD-WE}.  This makes the study of \eqref{SD-WE} harder than the zero exterior condition. To overcome this difficulty, we shall exploit a new technique which has been developed by Warma in \cite{CLR-MW,War-ACE,WZ} to solve fractional diffusion equations, fractional super diffusive equation and  strong damping wave equations. This original method shall allow us not only to prove well-posedness but also to have an explicit representation of solutions in terms of series which is crucial for the analysis of the controllability of the system.

To summarize, the main novelties of the present paper can be formulated as follows.

\begin{enumerate}
\item[(1)] For the first time, a {\it nonlocal} version of the MGT equation associated with the fractional Laplace operator with non-zero exterior condition has been studied.
 Some well-posedness results and an explicit representation of solutions in terms of series of the nonhomogeneous exterior value nonlocal evolution system \eqref{SD-WE} have been established.

\item[(2)] We have shown that the system is not null or exact controllable at time $T>0$.

\item[(3)] The unique continuation property of solutions to the adjoint system asso\-cia\-ted with \eqref{SD-WE} has been established. This result is obtained by carefully exploiting the unique continuation property for the eigenvalues problem of $(-\Delta)_D^s$ recently obtained in \cite{War-ACE} and by using some powerful tools from PDEs and complex analysis.

\item[(4)] The final important result is the approximate controllability of the system which is a direct consequence of the unique continuation property of the dual system.
\end{enumerate}


 Fractional order operators have emerged as a modeling alternative in various branches of science. For instance,
a number of stochastic models for explaining anomalous diffusion have been
introduced in the literature; among them we  quote the fractional Brownian motion; the continuous time random walk;  the L\'evy flights; the Schneider grey Brownian motion; and more generally, random walk models based on evolution equations of single and distributed fractional order in  space (see e.g. \cite{DS,GR,Man,Sch,ZL}).  In general, a fractional diffusion operator corresponds to a diverging jump length variance in the random walk. We refer to \cite{NPV,Val} and the references therein for a complete analysis, the derivation and the applications of the fractional Laplace operator.
For further details we also refer to \cite{GW-F,GW-CPDE} and their references.

The remaining of the paper is structured as follows. In Section \ref{sec-2} we state the main results of the article. The first one (Theorem \ref{lact-nul-cont}) says that if $b>0$ and $\alpha-\frac{c^2}{b}>0$, then the system \eqref{SD-WE} is not exact or null controllable at time $T>0$.  Our second main result (Theorem \ref{pro-uni-con}) is the unique continuation property for the adjoint system associated with \eqref{SD-WE}.
 The third main result (Theorem \ref{main-Theo}) states that the system is approximately controllable at any $T>0$ and for every $g\in\mathcal D(\mathcal O\times(0,T))$. This last result will be obtained as a direct consequence of the above mentioned unique continuation property. In Section \ref{preli} we introduce the function spaces needed throughout the paper, give a rigorous definition of the fractional Laplacian and some known results that will be used in the proofs of our main results. The proofs of the well-posedness and an explicit representation of solutions to the system \eqref{SD-WE} and the associated dual system are contained in Section \ref{sec-4}. Finally, in Section \ref{prof-ma-re} we give the proofs of our main results on controllability. 

\section{Main results}\label{sec-2}

In this section we state the main results of the article. Throughout the remainder of the paper, without any mention, $\alpha$, $b$, $c$ and $0<s<1$ are real numbers and $\Omega\subset\RR^N$ denotes a bounded open set with a Lipschitz continuous boundary. Given a measurable set $E\subset\RR^N$, we shall denote by $(\cdot,\cdot)_{L^2(E)}$ the scalar product in $L^2(E)$.  We refer to Section \ref{preli} for a rigorous definition of the function spaces and operators involved. Let $W_0^{s,2}(\bOm)$ be the energy space and denote by $W^{-s,2}(\bOm)$ its dual. We shall denote by  $\langle\cdot,\cdot\rangle_{-\frac 12,\frac 12}$ the duality pair between $W^{-s,2}(\bOm)$ and $W_0^{s,2}(\bOm)$ (see Section \ref{preli}).



Next, we introduce our notion of solution.
Let  $(u_0,u_1,u_2)\in L^2(\Omega)\times W^{-s,2}(\bOm)\times W^{-s,2}(\bOm)$ and consider the following two systems:
\begin{equation}\label{main-EQ-2}
\begin{cases}
v_{ttt}+\alpha v_{tt}+c^2(-\Delta)^sv + b(-\Delta)^sv_t=0\;\;&\mbox{ in }\; \Omega\times (0,T),\\
v=g&\mbox{ in }\;(\Omc)\times (0,T),\\
v(\cdot,0)=0,\;\;v_t(\cdot,0)=0,\;\;v_{tt}(\cdot,0)=0 &\mbox{ in }\;\Omega,
\end{cases}
\end{equation}
and
\begin{equation}\label{main-EQ-3}
\begin{cases}
w_{ttt}+\alpha w_{tt}+c^2(-\Delta)^sw+b(-\Delta)^sw_t=0\;\;&\mbox{ in }\; \Omega\times (0,T),\\
w=0&\mbox{ in }\;(\Omc)\times (0,T),\\
w(\cdot,0)=u_0,\;\;w_t(\cdot,0)=u_1,\;\; w_{tt}(\cdot,0)=u_2&\mbox{ in }\;\Omega.
\end{cases}
\end{equation}
Then it is clear that $u=v+w$ solves the system \eqref{SD-WE}. We next introduce our notion of weak solution to the system \eqref{main-EQ-2}.


\begin{definition}\label{def-strong-sol}
Let  $g$ be a given function. 
A function $(v,v_t,v_{tt})$ is said to be a weak solution of \eqref{main-EQ-2}, if the following properties hold.

\begin{itemize}
\item Regularity:
\begin{equation}\label{regu}
\begin{cases}
v\in C([0,T];L^2(\Omega))\cap  C^2([0,T];W^{-s,2}(\bOm)), \\
v_{ttt}\in C((0,T); W^{-s,2}(\bOm)).
\end{cases}
\end{equation}
\item Variational identity: For every $w\in W_0^{s,2}(\bOm)$ and a.e. $t\in (0,T)$,
\begin{align*}
\langle v_{ttt}+\alpha v_{tt},w\rangle_{-\frac 12,\frac 12}+\langle (-\Delta)^s(c^2v+b v_t),w\rangle_{-\frac 12,\frac 12}=0.
\end{align*}  
\item Initial and exterior conditions: 
\begin{align}\label{Var-I}
v(\cdot,0)=0, \;v_t(\cdot,0)=0,  \;v_{tt}(\cdot,0)=0\;\;\mbox{ in }\;\Omega\;\mbox{ and }\; v=g\;\mbox {in }\;(\Omc)\times (0,T).
\end{align}
\end{itemize}
\end{definition}

By Definition \ref{def-strong-sol}, for a weak solution $(v,v_t,v_{tt})$ of  \eqref{main-EQ-2}, we have that

\[(v(\cdot,T),v_t(\cdot,T),v_{tt}(\cdot,T))\in  L^2(\Omega)\times W^{-s,2}(\bOm)\times W^{-s,2}(\bOm).\]

Concerning existence and uniqueness of weak solutions to the system \eqref{SD-WE}  we have the following result.

\begin{theorem}\label{th:2.3}
For  every $(u_0, u_1,u_2)\in W_0^{s,2}(\bOm)\times W_0^{s,2}(\bOm)\times L^2(\Omega)$ and $g\in\mathcal D((\Omc)\times(0,T))$, the system \eqref{SD-WE} has a unique weak solution $(u,u_t,u_{tt})$ given by
 \begin{multline*}
u(x,t)=\sum_{n=1}^{\infty}\Big(A_{n}(t)u_{0,n}+B_{n}(t)u_{1,n}+C_n(t)u_{2,n}\Big)\varphi_{n}(x)\notag\\
+\sum_{n=1}^{\infty}\left(\int_0^t \Big( g(\cdot,\tau),\mathcal{N}_{s}\varphi_{n} \Big)_{L^2(\Omc)}\frac{1}{\lambda_{n}}\Big(C_{n}'''(t-\tau)+\alpha C_{n}''(t-\tau)\Big)d\tau\right)\varphi_{n}(x),
\end{multline*}
where $(\varphi_n)_{n\in\NN}$ is the orthornormal basis of eigenfunctions of the operator $(-\Delta)_D^s$ associated with the eigenvalues $(\lambda_n)_{n\in\NN}$ and $A_{n}(t), B_{n}(t), C_{n}(t)$ are defined in Proposition \ref{pro-sol-SG} below.
\end{theorem}

Our first main result concerning controllability says that if $b>0$ and $\alpha-\frac{c^2}{b}>0$, then the system \eqref{SD-WE} is not exact or null controllable.

\begin{theorem}\label{lact-nul-cont}
Let $b>0$ and $\alpha-\frac{c^2}{b}>0$. Then  the system \eqref{SD-WE} is not exact or null controllable at time $T>0$.
\end{theorem}

Since \eqref{SD-WE} is not exact or null controllable if $b>0$ and and $\alpha-\frac{c^2}{b}>0$, then we shall study if it can be approximately controllable.
It is straightforward to verify that the study of the approximate controllability of  \eqref{SD-WE} can be reduced to the case $u_0=u_1=u_2=0$.  We refer to \cite{KW,CLR-MW,rosier2007,War-AA,War-ACE,Zua1} for more details. 

From now on, without any mention, we shall assume that

\begin{align*}
b>0\;\mbox{ and }\; \alpha-\frac{c^2}{b}>0.
\end{align*}

Using the classical integration by parts formula, we have that the following backward system 
\begin{equation}\label{ACP-Dual}
\begin{cases}
-\psi_{ttt} +\alpha \psi_{tt}+c^2(-\Delta)^s\psi-b(\Delta)^s\psi_t=0\;\;&\mbox{ in }\; \Omega\times (0,T),\\
\psi=0&\mbox{ in }\;(\Omc)\times (0,T),\\
\psi(\cdot,T)=\psi_0,\;\; \psi_t(\cdot,T)=-\psi_1,\;\; \psi_{tt}(\cdot,T)=\psi_2&\mbox{ in }\;\Omega,
\end{cases}
\end{equation}
can be viewed as the dual system associated with \eqref{main-EQ-2}.
Our notion of weak solution to \eqref{ACP-Dual} is as follows.

\begin{definition}
Let $(\psi_0,\psi_1,\psi_2) \in W_0^{s,2}(\bOm)\times W_0^{s,2}(\bOm)\times L^2(\Omega)$. 
A function $(\psi,\psi_t,\psi_{tt})$ is said to be a weak  solution of \eqref{ACP-Dual}, if for a.e. $t\in (0,T)$,  the following properties hold.

\begin{itemize}
\item Regularity and final data:
\begin{equation}\label{Dual-egu}
\begin{cases}
\psi\in C^1([0,T];W_0^{s,2}(\bOm))\cap C^2([0,T]; L^2(\Omega)), \\
 \psi_{ttt}\in C((0,T);W^{-s,2}(\bOm)),
 \end{cases}
\end{equation}
$\psi(\cdot,T)=\psi_0$, $\psi_t(\cdot,T)=\psi_1$ and $\psi_{tt}(\cdot,T)=\psi_2$ in $\Omega$.
\item Variational  identity: For every $w\in  W_0^{s,2}(\bOm)$ and a.e. $t\in (0,T)$, 
\begin{align*}
\langle \psi_{ttt}+\alpha \psi_{tt},w\rangle_{-\frac 12,\frac 12}+\langle(-\Delta)^s(c^2\psi-b\psi_t),w\rangle_{-\frac 12,\frac 12}=0.
\end{align*}  
\end{itemize}
\end{definition}

We have the following existence result.

\begin{theorem}\label{theo-48}
For every $(\psi_0,\psi_1,\psi_2)\in W_0^{s,2}(\bOm)\times W_0^{s,2}(\bOm)\times L^2(\Omega)$, the dual system \eqref{ACP-Dual} has a unique weak solution $(\psi,\psi_t,\psi_{tt})$ given by
\begin{align}\label{eq-25}
\psi(x,t)=\sum_{n=1}^{\infty}\Big(\psi_{0,n}A_{n}(T-t)- \psi_{1,n}B_{n}(T-t)+\psi_{2,n}C_{n}(T-t)\Big)\varphi_{n}(x),
\end{align}
where  $A_n(t)$, $B_n(t)$ and $C_n(t)$ are given in \eqref{An}, \eqref{Bn} and \eqref{Cn}, respectively. In addition the following assertions hold.
\begin{enumerate}
\item There is a constant $C>0$ such that for all $t\in [0,T]$,
\begin{multline}\label{Dual-EST-1}
 \|\psi(\cdot,t)\|_{W_0^{s,2}(\bOm)}^2+ \|\psi_{t}(\cdot,t)\|_{W_0^{s,2}(\bOm)}^2+ \|\psi_{tt}(\cdot,t)\|_{L^2(\Omega)}^2\\\le C\left(\|\psi_0\|_{W_0^{s,2}(\bOm)}^2+\|\psi_1\|_{W_0^{s,2}(\bOm)}^2+\|\psi_2\|_{L^2(\Omega)}^2\right),
\end{multline}
and 
\begin{equation}\label{Dual-EST-1-2}
 \|\psi_{ttt}(\cdot,t)\|_{W^{-s,2}(\bOm)}^2\le \left(\|\psi_0\|_{W_0^{s,2}(\bOm)}^2+\|\psi_1\|_{W_0^{s,2}(\bOm)}^2+\|\psi_2\|_{L^2(\Omega)}^2\right).
 \end{equation}
\item We have that $\psi\in L^\infty((0,T); D((-\Delta)_D^s))$.
 
\item The mapping 
\[[0,T)\ni t\mapsto\mathcal N_s\psi(\cdot,t)\in L^2(\Omc),\]
 can be analytically extended to the half-plane $\Sigma_T:=\{z\in\CC:\;{Re}(z)<T\}$. Here, $\mathcal N_s\psi$ is the nonlocal normal derivative of $\psi$ defined in \eqref{NLND} below.

\end{enumerate}
\end{theorem}

The next result says that the adjoint system \eqref{ACP-Dual} satisfies the {\em unique conti\-nua\-tion property for evolution equations} which is our second main result.

\begin{theorem}\label{pro-uni-con}
Let $(\psi_0,\psi_1,\psi_2)\in W_0^{s,2}(\bOm)\times W_0^{s,2}(\bOm)\times L^2(\Omega)$ and let $(\psi,\psi_t,\psi_{tt})$ be the unique weak solution of \eqref{ACP-Dual}. Let $\mathcal O\subset\Omc$ be an arbitrary non-empty open set. If $\mathcal N_s\psi=0$ in $\mathcal O\times (0,T)$, then $\psi=0$ in $\Omega\times (0,T)$. \end{theorem}

The last main result concerns the approximate controllability of \eqref{SD-WE}.

\begin{theorem}\label{main-Theo}
The system \eqref{SD-WE} is approximately controllable for any $T>0$ and $g\in \mathcal D(\mathcal O\times(0,T))$, where $\mathcal O\subset\Omc$ is an arbitrary non-empty open set. That is,
\begin{align*}
\overline{\mathcal R((0,0,0),T)}^{L^2(\Omega)\times W^{-s,2}(\bOm)\times W^{-s,2}(\bOm)}=
L^2(\Omega)\times W^{-s,2}(\bOm)\times W^{-s,2}(\bOm),
\end{align*}
where $(u,u_t,u_{tt})$ is the unique weak solution of \eqref{SD-WE} with $u_0=u_1=u_2=0$.
\end{theorem}

\section{Preliminaries}\label{preli}

In this section we give some notations and recall some known results as they are needed in the proof of our main results.

We start with the fractional order Sobolev spaces. Given $0<s<1$,  we let
\begin{align*}
W^{s,2}(\Omega):=\left\{u\in L^2(\Omega):\;\int_{\Omega}\int_{\Omega}\frac{|u(x)-u(y)|^2}{|x-y|^{N+2s}}\;dxdy<\infty\right\},
\end{align*}
and we endow it with the norm 
\begin{align*}
\|u\|_{W^{s,2}(\Omega)}:=\left(\int_{\Omega}|u(x)|^2\;dx+\int_{\Omega}\int_{\Omega}\frac{|u(x)-u(y)|^2}{|x-y|^{N+2s}}\;dxdy\right)^{\frac 12}.
\end{align*}
We set
\begin{align*}
W_0^{s,2}(\bOm):=\Big\{u\in W^{s,2}(\RR^N):\;u=0\;\mbox{ in }\;\RR^N\setminus\Omega\Big\}.
\end{align*}

For more information on fractional order Sobolev spaces, we refer to \cite{NPV,Gris,JW,War}.

Next, we give a rigorous definition of the fractional Laplace operator. Let 
\begin{align*}
\mathcal L_s^{1}(\RR^N):=\left\{u:\RR^N\to\RR\;\mbox{ measurable},\; \int_{\RR^N}\frac{|u(x)|}{(1+|x|)^{N+2s}}\;dx<\infty\right\}.
\end{align*}
For $u\in \mathcal L_s^{1}(\RR^N)$ and $\varepsilon>0$ we set
\begin{align*}
(-\Delta)_\varepsilon^s u(x):= C_{N,s}\int_{\{y\in\RR^N:\;|x-y|>\varepsilon\}}\frac{u(x)-u(y)}{|x-y|^{N+2s}}\;dy,\;\;x\in\RR^N,
\end{align*}
where $C_{N,s}$ is a normalization constant given by
\begin{align}\label{CNs}
C_{N,s}:=\frac{s2^{2s}\Gamma\left(\frac{2s+N}{2}\right)}{\pi^{\frac
N2}\Gamma(1-s)}.
\end{align}
The {\em fractional Laplacian}  $(-\Delta)^s$ is defined by the following singular integral:
\begin{align}\label{fl_def}
(-\Delta)^su(x):=C_{N,s}\;\mbox{P.V.}\int_{\RR^N}\frac{u(x)-u(y)}{|x-y|^{N+2s}}dy=\lim_{\varepsilon\downarrow 0}(-\Delta)_\varepsilon^s u(x),\;x\in\RR^N,
\end{align}
provided that the limit exists. 
We notice that $\mathcal L_s^{1}(\RR^N)$ is the right space for which $ v:=(-\Delta)_\varepsilon^s u$ exists for every $\varepsilon>0$, $v$ being also continuous at the continuity points of  $u$.  The fractional Laplacian can be also defined as the pseudo-differential operator with symbol $|\xi|^{2s}$.
For more details on the fractional Laplace operator we refer to \cite{BBC,Caf1,NPV,GW,GW-CPDE,GW2,War,War-In} and their references.


Next, we consider the following Dirichlet problem:
\begin{equation}\label{EDP}
\begin{cases}
(-\Delta)^su=0\;\;&\mbox{ in }\;\Omega,\\
u=g&\mbox{ in }\;\RR^N\setminus\Omega.
\end{cases}
\end{equation}

\begin{definition}\label{def-sol}
Let $g\in W^{s,2}(\Omc)$ and  $\widetilde g\in W^{s,2}(\RR^N)$ be such that $\widetilde g|_{\Omc}=g$. A  $u\in W^{s,2}(\RR^N)$ is said to be a weak solution of \eqref{EDP} if $u-\widetilde g\in W_0^{s,2}(\bOm)$ and 
\begin{align*}
\int_{\RR^N}\int_{\RR^N}\frac{(u(x)-u(y))(v(x)-v(y))}{|x-y|^{N+2s}}\;dxdy=0,\;\;\forall\; v\in W_0^{s,2}(\bOm).
\end{align*}
\end{definition}

The following existence result is taken from \cite{Grub} (see also \cite{GSU}).

\begin{proposition}\label{proposi-33}
For every $g\in W^{s,2}(\RR^N\setminus\Omega)$, there is a unique $u\in W^{s,2}(\RR^N)$ satisfying \eqref{EDP} in the sense of Definition \ref{def-sol}. In addition, there is a constant $C>0$ such that
\begin{align}\label{E-Neu}
\|u\|_{W^{s,2}(\RR^N)}\le C\|g\|_{W^{s,2}(\RR^N\setminus\Omega)}.
\end{align}
\end{proposition}

Now, we consider the realization of $(-\Delta)^s$ in $L^2(\Omega)$ with the condition $u=0$ in $\Omc$. More precisely, we consider the closed and bilinear form
\begin{align*}
\mathcal F(u,v):=\frac{C_{N,s}}{2}\int_{\RR^N}\int_{\RR^N}\frac{(u(x)-u(y))(v(x)-v(y))}{|x-y|^{N+2s}}\;dxdy,\;\;u,v\in W_0^{s,2}(\bOm).
\end{align*}
Let $(-\Delta)_D^s$ be the selfadjoint operator in $L^2(\Omega)$ associated with $\mathcal F$ in the sense that
\begin{equation*}
\begin{cases}
D((-\Delta)_D^s)=\Big\{u\in W_0^{s,2}(\bOm),\;\exists\;f\in L^2(\Omega),\;\mathcal F(u,v)=(f,v)_{L^2(\Omega)}\;\forall\;v\in W_0^{s,2}(\bOm)\Big\},\\
(-\Delta)_D^su=f.
\end{cases}
\end{equation*}
More precisely, we have that
\begin{equation}\label{DLO}
D((-\Delta)_D^s):=\left\{u\in W_0^{s,2}(\bOm),\; (-\Delta)^su\in L^2(\Omega)\right\},\;\;\;
(-\Delta)_D^su:=(-\Delta)^su.
\end{equation}
Then $(-\Delta)_D^s$ is the realization of $(-\Delta)^s$ in $L^2(\Omega)$ with the condition $u=0$ in $\Omc$. It is well-known (see e.g. \cite{Val,War-ACE}) that $(-\Delta)_D^s$ has a compact resolvent and its eigenvalues form a non-decreasing sequence of real numbers $0<\lambda_1\le\lambda_2\le\cdots\le\lambda_n\le\cdots$ satisfying $\lim_{n\to\infty}\lambda_n=\infty$.  In addition, the eigenvalues are of finite multiplicity.
Let $(\varphi_n)_{n\in\NN}$ be the orthonormal basis of eigenfunctions associated with  $(\lambda_n)_{n\in\NN}$. Then $\varphi_n\in D((-\Delta)_D^s)$ for every $n\in\NN$,  $(\varphi_n)_{n\in\NN}$ is total in $L^2(\Omega)$ and satisfies 
\begin{equation}\label{ei-val-pro}
\begin{cases}
(-\Delta)^s\varphi_n=\lambda_n\varphi_n\;\;&\mbox{ in }\;\Omega,\\
\varphi_n=0\;&\mbox{ in }\;\Omc.
\end{cases}
\end{equation}
With this setting, we have that for $\gamma\ge 0$, we can define the $\gamma$-powers of $(\Delta)_D^s$ as follows:
\begin{align}\label{FP}
\begin{cases}
D([(-\Delta)_D^s]^\gamma):=\left\{u\in L^2(\Omega):\; \sum_{n=1}^\infty\left|\lambda_n^\gamma(u,\varphi_n)_{L^2(\Omega)}\right|^2<\infty\right\},\vspace*{0.1cm}\\
[(-\Delta)_D^s]^\gamma u:=\sum_{n=1}^\infty\lambda_n^\gamma(u,\varphi_n)_{L^2(\Omega)}.
\end{cases}
\end{align}
Using  \eqref{FP}, we can easily show that $D([(-\Delta)_D^s]^{\frac 12})=W_0^{s,2}(\bOm)$ and for $u\in W_0^{s,2}(\bOm)$ we have that
\begin{align}\label{norm-V}
\|u\|_{W_0^{s,2}(\bOm)}^2=\sum_{n=1}^\infty\left|\lambda_n^{\frac 12}\left(u,\varphi_n\right)_{L^2(\Om)}\right|^2,
\end{align}
defines an equivalent norm on $W_0^{s,2}(\bOm)$. If $u\in D((-\Delta)_D^s)$, then
\begin{align*}
\|u\|_{D((-\Delta)_D^s)}^2=\|(-\Delta)_D^su\|_{L^2(\Omega)}^2=\sum_{n=1}^\infty\left|\lambda_n\left(u,\varphi_n\right)_{L^2(\Om)}\right|^2.
\end{align*}
In addition, for $u\in W^{-s,2}(\bOm)$, we have that
\begin{align}\label{norm-V-2}
\|u\|_{W^{-s,2}(\bOm)}^2=\sum_{n=1}^\infty\left|\lambda_n^{-\frac 12}\left(u,\varphi_n\right)_{L^2(\Om)}\right|^2.
\end{align}  
In that case, using the so called Gelfand triple (see e.g. \cite{ATW}) we have the following continuous embeddings $W_0^{s,2}(\bOm)\hookrightarrow L^2(\Omega)\hookrightarrow W^{-s,2}(\bOm)$.

Next, for $u\in W^{s,2}(\RR^N)$ we introduce the {\em nonlocal normal derivative $\mathcal N_s$} given by 
\begin{align}\label{NLND}
\mathcal N_su(x):=C_{N,s}\int_{\Omega}\frac{u(x)-u(y)}{|x-y|^{N+2s}}\;dy,\;\;\;x\in\RR^N\setminus\bOm,
\end{align}
where $C_{N,s}$ is the constant given in \eqref{CNs}.
By \cite[Lemma 3.2]{GSU},  for every $u\in W^{s,2}(\RR^N)$, we have that $\mathcal N_su\in L^2(\Omc)$.

The following unique continuation property which shall play an important role in the proof of our main results has been recently obtained in \cite[Theorem 3.10]{War-ACE}.

\begin{lemma}\label{lem-UCD}
Let $\lambda>0$ be a real number  and $\mathcal O\subset\Omb$ a non-empty open set. 
If $\varphi\in D((-\Delta)_D^s)$ satisfies
\begin{equation*}
(-\Delta)_D^s\varphi=\lambda\varphi\;\mbox{ in }\;\Omega\;\mbox{ and }\; \mathcal N_s\varphi=0\;\mbox{ in }\;\mathcal O,
\end{equation*}
then $\varphi=0$ in $\RR^N$. 
\end{lemma}

\begin{remark}
{\em The following important identity has been recently proved in \cite[Reamrk 3.11]{War-ACE}.
Let $g\in W^{s,2}(\RR^N\setminus\Omega)$ and $U_g\in W^{s,2}(\RR^N)$ the associated unique weak solution of the Dirichlet problem \eqref{EDP}. Then
\begin{align}\label{eqA9}
\int_{\Omc}g\mathcal N_s\varphi_n\;dx=-\lambda_n\int_{\Omega}\varphi_nU_g\;dx.
\end{align}
}
\end{remark}

For more details on the Dirichlet problem associated with the fractional Laplace operator we refer the interested reader to \cite{BWZ2,BWZ1,BBC,Caf1,Grub,RS2-2,RS-DP,War,War-ACE} and their references.

The following integration by parts formula is contained in \cite[Lemma 3.3]{DRV} for smooth functions. The version given here can be obtained by using a simple density argument (see e.g. \cite{War-ACE}).

\begin{proposition}
 Let $u\in W^{s,2}(\RR^N)$ be such that $(-\Delta)^su\in L^2(\Omega)$. Then for every $v\in W^{s,2}(\RR^N)$, we have
\begin{align}\label{Int-Part}
\frac{C_{N,s}}{2}\int\int_{\RR^{2N}\setminus(\Omc)^2}&\frac{(u(x)-u(y))(v(x)-v(y))}{|x-y|^{N+2s}}\;dxdy\notag\\
=&\int_{\Omega}v(-\Delta)^su\;dx+\int_{\Omc}v\mathcal N_su\;dx.
\end{align}
\end{proposition}

We conclude this section with the following observation.

\begin{remark}
{\em We mention the following facts.
\begin{enumerate}
\item Firstly, we notice that 
\begin{align*}
\RR^{2N}\setminus(\RR^N\setminus\Omega)^2=(\Omega\times\Omega)\cup (\Omega\times(\RR^N\setminus\Omega))\cup((\RR^N\setminus\Omega)\times\Omega).
\end{align*}

\item Secondly, if $u=0$ in $\Omc$ or $v=0$ in $\Omc$, then
\begin{align*}
\int\int_{\RR^{2N}\setminus(\RR^N\setminus\Omega)^2}\frac{(u(x)-u(y))(v(x)-v(y))}{|x-y|^{N+2s}}dxdy=\int_{\RR^N}\int_{\RR^N}\frac{(u(x)-u(y))(v(x)-v(y))}{|x-y|^{N+2s}}dxdy,
\end{align*}
so that for such functions, the identity \eqref{Int-Part} becomes
\begin{align}\label{Int-Part-2}
\frac{C_{N,s}}{2}\int_{\RR^N}\int_{\RR^N}\frac{(u(x)-u(y))(v(x)-v(y))}{|x-y|^{N+2s}}\;dxdy
=\int_{\Omega}v(-\Delta)^su\;dx+\int_{\Omc}v\mathcal N_su\;dx.
\end{align}
\end{enumerate}
}
\end{remark}

%
%
%

\section{Series representation of solutions}\label{sec-4}

In this section we prove a representation in terms of series of weak solutions to the system \eqref{main-EQ-2} and the dual system \eqref{ACP-Dual}. Evolution equations with non-homogeneous boundary or exterior conditions are in general not so easy to solve since one cannot apply directly semigroup methods due the fact that the associated operator is in general not a generator of a semigroup. For this reason, we shall give more details in the proofs.
The representation of solutions in term of series shall play an important role in the proofs of our main results. 

Throughout the remainder of the article, without any mention, we shall denote by $(\varphi_n)_{n\in\NN}$ the orthornormal basis of eigenfunctions of the operator $(-\Delta)_D^s$ associated with the eigenvalues $(\lambda_n)_{n\in\NN}$. 

We also recall that $b>0$ and $\alpha-\frac{c^2}{b}>0$,

\subsection{Series solutions of the system \eqref{main-EQ-2}}

We have shown in Section \ref{sec-2} that a solution $(u,u_t,u_{tt})$ of \eqref{SD-WE} can be written as $u=v+w$ where $(v,v_t,v_{tt})$ solves \eqref{main-EQ-2} and $(w,w_t,w_{tt})$ is a solution of \eqref{main-EQ-3}.
 
Consider the system \eqref{main-EQ-3} which is equivalent to the following system:
\begin{equation}\label{main-EQ-3-2}
\begin{cases}
w_{ttt}+\alpha w_{tt}+c^2(-\Delta)_D^sw+b(-\Delta)_D^sw_t=0\;\;&\mbox{ in }\; \Omega\times (0,T),\\
w(\cdot,0)=u_0,\;\;w_t(\cdot,0)=u_1,\;\; w_{tt}(\cdot,0)=u_2&\mbox{ in }\;\Omega.
\end{cases}
\end{equation} 
Let 
\begin{equation*}
W=\left(
\begin{array}{c}
w\\w_t \\w_{tt}\\ 
\end{array}
\right)\;\;\mbox{ and }\;   W_0=\left(
\begin{array}{c}
u_0\\u_1 \\ u_2\\
\end{array}
\right).
\end{equation*}
Then \eqref{main-EQ-3-2} can be rewritten as the following first order Cauchy problem:
\begin{equation}\label{equa27}
\begin{cases}
W_{t}+\mathcal A W=0\;\;&\mbox{ in }\; \Omega\times (0,T),\\
W(\cdot,0)=W_0&\mbox{ in }\;\Omega,
\end{cases}
\end{equation}
where the operator matrix $\mathcal A$ with domain $D(\mathcal A)=D((-\Delta)_D^s)\times D((-\Delta)_D^s)\times L^2(\Omega)$ is given by 
\begin{equation}\label{eq-42}
\mathcal A:=\left(
\begin{array}{ccc}
0& -I & 0\\
0 & 0 & -I\\
c^2(-\Delta)_D^s& b(-\Delta)_D^s & \alpha I
\end{array}
\right).
\end{equation}

Let $D([(-\Delta)_D^s]^{\frac 12})$ be the space defined in \eqref{FP} and let
\begin{align*}
\mathcal H:=D([(-\Delta)_D^s]^{\frac 12})\times D([(-\Delta)_D^s]^{\frac 12})\times L^2(\Om)=W_0^{s,2}(\bOm)\times W_0^{s,2}(\bOm)\times L^2(\Omega)
\end{align*}
be endowed with the graph norm
\begin{align*}
\|U_0\|_{\mathcal H}^2=\|[(-\Delta)_D^2]^{\frac 12}u_0\|_{L^2(\Omega)}^2+\|[(-\Delta)_D^2]^{\frac 12}u_1\|_{L^2(\Omega)}^2+\|u_2\|_{L^2(\Omega)}^2.
\end{align*}

Noticing that the operator $(\Delta)_D^s$ enters in the framework of \cite{KLM}, we have the following result taken from  \cite[Theorem 1.1]{KLM} which has been proved with a generic operator $A$.

\begin{theorem}
The operator $-\mathcal A$ generates a strongly continuous group in $\mathcal H$. As a consequence, for every $W_0:=(u_0,u_1,u_2)\in\mathcal H$, the system  \eqref{equa27}, hence, the system \eqref{main-EQ-3-2}, has a unique strong solution $W$ given by $W(t)=e^{-tA}W_0$, where $(e^{-tA})_{t\ge 0}$ is the strongly continuous semigroup generated by the operator $-A$.
\end{theorem}

Knowing that the system is well-posed, we are interested to have an explicit representation of solutions which is crucial for the study of the controllability of the system.

From the work of Marchand, McDevitt and Triggiani \cite{marchand2012abstract}, we have the following result.

\begin{lemma}\label{eigen-values-functions}
Each pair $\{\lambda_{n},\varphi_{n}\}$ of eigenvalues and eigenfunctions of $(-\Delta)_{D}^s$ generates eigenvalues $\{\lambda_{n,j}\}_{n\in\NN}$, $j=1,2,3$, of $\mathcal{A}$ given as the roots of the following cubic equation:
\begin{align}\label{charact-eq}
\lambda_{n,j}^3+\alpha \lambda_{n,j}^2+(\lambda_{n}b)\lambda_{n,j}+\lambda_{n}c^2=0.
\end{align}
Besides, under the condition that $\gamma:=\alpha-\frac{c^2}{b}>0$, one root $\lambda_{n,1}$ is real and the other two $\lambda_{n,2}$ and $\lambda_{n,3}$ are complex conjugates, all with negative real parts.
Moreover, the eigenvalues satisfy the following asymptotic behavior:
\begin{align}\label{asym-eigenvalues}
\begin{cases}
\lim_{n\to\infty}\lambda_{n,1}=-\frac{c^2}{b},\quad \lim_{n\to\infty}\frac{\lambda_{n}}{(\operatorname{Re }\lambda_{n,2})^2}=\frac{1}{b},\quad \operatorname{Re }\lambda_{n,2}\searrow -\frac{\gamma}{2},\vspace*{0.1cm}\\ |\operatorname{Im }\lambda_{n,2}|\sim \sqrt{b\lambda_{n}}\to\infty,\; \text{as }n\to\infty.
\end{cases}
\end{align}
\end{lemma}

Throughout the rest of the article we assume that 
\begin{align*}
\gamma=\alpha-\frac{c^2}{b}>0.
\end{align*}

Next we give the representation of solutions in terms of series.

\begin{proposition}\label{pro-sol-SG}
Let  $(u_0, u_1,u_2)\in W_0^{s,2}(\bOm)\times W_0^{s,2}(\bOm)\times L^2(\Omega)$. Then the unique solution $(w,w_t,w_{tt})$ of the system \eqref{main-EQ-3-2} is given by
\begin{align}\label{eq-9}
w(x,t)=\sum_{n=1}^{\infty}\Big(A_{n}(t)(u_{0},\varphi_{n})_{L^2(\Om)}+B_{n}(t)(u_{1},\varphi_{n})_{L^2(\Om)}+C_{n}(t)(u_{2},\varphi_{n})_{L^2(\Om)}\Big)\varphi_{n}(x),
\end{align}
where
\begin{align}\label{An}
A_{n}(t)=\frac{\lambda_{n,2}\lambda_{n,3}}{\xi_{n,1}}e^{\lambda_{n,1}t}-\frac{\lambda_{n,1}\lambda_{n,3}}{\xi_{n,2}}e^{\lambda_{n,2}t}+\frac{\lambda_{n,1}\lambda_{n,2}}{\xi_{n,3}}e^{\lambda_{n,3}t},
\end{align}
\begin{align}\label{Bn}
B_{n}(t)=-\frac{\lambda_{n,2}+\lambda_{n,3}}{\xi_{n,1}}e^{\lambda_{n,1}t}+\frac{\lambda_{n,1}+\lambda_{n,3}}{\xi_{n,2}}e^{\lambda_{n,2}t}-\frac{\lambda_{n,1}+\lambda_{n,2}}{\xi_{n,3}}e^{\lambda_{n,3}t},
\end{align}
and
\begin{align}\label{Cn}
C_{n}(t)=\frac{1}{\xi_{n,1}}e^{\lambda_{n,1}t}-\frac{1}{\xi_{n,2}}e^{\lambda_{n,2}t}+\frac{1}{\xi_{n,3}}e^{\lambda_{n,3}t}.
\end{align}
Here, $\lambda_{n,j}$ are the solutions of \eqref{charact-eq} and
\begin{align}\label{xi}
\xi_{n,1}&=(\lambda_{n,1}-\lambda_{n,2})(\lambda_{n,1}-\lambda_{n,3}),\quad \xi_{n,2}=(\lambda_{n,1}-\lambda_{n,2})(\lambda_{n,2}-\lambda_{n,3}),\notag\\ \xi_{n,3}&=(\lambda_{n,1}-\lambda_{n,3})(\lambda_{n,2}-\lambda_{n,3}).
\end{align}
\end{proposition}

\begin{proof}
 Using the spectral theorem of selfadjoint operators, we can proceed with the method of separation of variables. That is, we look for a solution $(w,w_t,w_{tt})$ of \eqref{main-EQ-3-2} in the form

\begin{align}\label{3}
w(x,t)=\sum_{n=1}^{\infty}\left( w(\cdot,t),\varphi_{n}\right)_{L^2(\Omega)} \varphi_{n}(x).
\end{align}
For the sake of simplicity we let $w_{n}(t)=\left( w(\cdot,t),\varphi_{n}\right)_{L^2(\Omega)}$. Replacing \eqref{3} in the first equation of \eqref{main-EQ-3-2}, then multiplying both sides with $\varphi_k$ and integrating over $\Omega$, we get that $w_n(t)$ satisfies the following ordinary differential equation (ODE)
\begin{align}\label{4}
w_{n}'''(t)+\alpha w_{n}''(t)+c^2\lambda_{n}w_{n}(t)+b\lambda_{n}w_{n}'(t)=0.
\end{align}
From Lemma \ref{eigen-values-functions} and letting $u_{0,n}=(u_0,\varphi_n)_{L^2(\Om)}$, $u_{1,n}=(u_1,\varphi_n)_{L^2(\Om)}$ and $u_{2,n}=(u_2,\varphi_n)_{L^2(\Om)}$ so that
\begin{align}\label{6}
u_0=\sum_{n=1}^{\infty}u_{0,n}\varphi_{n}, \quad u_1=\sum_{n=1}^{\infty}u_{1,n}\varphi_{n},\quad u_2= \sum_{n=1}^{\infty}u_{2,n}\varphi_{n},
\end{align}
we get
\begin{align}\label{5}
w(x,t)=\sum_{n=1}^{\infty}\sum_{j=1}^3 a_{n,i}e^{\lambda_{n,i}t}\varphi_{n}(x),
\end{align}
where
\begin{align}\label{7}
a_{n,1}&=\frac{\lambda_{n,2}\lambda_{n,3}}{\xi_{n,1}}u_{0,n}-\frac{(\lambda_{n,2}+\lambda_{n,3})}{\xi_{n,1}}u_{1,n}+\frac{1}{\xi_{n,1}}u_{2,n},\\
a_{n,2}&=-\frac{\lambda_{n,1}\lambda_{n,3}}{\xi_{n,2}}u_{0,n}+\frac{(\lambda_{n,1}+\lambda_{n,3})}{\xi_{n,2}}u_{1,n}-\frac{1}{\xi_{n,2}}u_{2,n},\\
a_{n,3}&=\frac{\lambda_{n,1}\lambda_{n,2}}{\xi_{n,3}}u_{0,n}-\frac{(\lambda_{n,1}+\lambda_{n,2})}{\xi_{n,3}}u_{1,n}+\frac{1}{\xi_{n,3}}u_{2,n}.
\end{align}
Therefore, we obtain the following expression of $w$:
\begin{align}\label{9}
&w(x,t)=\sum_{n=1}^{\infty}\left[u_{0,n}\left(\frac{\lambda_{n,2}\lambda_{n,3}}{\xi_{n,1}}e^{\lambda_{n,1}t}-\frac{\lambda_{n,1}\lambda_{n,3}}{\xi_{n,2}}e^{\lambda_{n,2}t}+\frac{\lambda_{n,1}\lambda_{n,2}}{\xi_{n,3}}e^{\lambda_{n,3}t}\right)\right]\varphi_{n}(x)\notag\\
&+\sum_{n=1}^{\infty}\left[u_{1,n}\left(-\frac{(\lambda_{n,2}+\lambda_{n,3})}{\xi_{n,1}}e^{\lambda_{n,1}t}+\frac{(\lambda_{n,1}+\lambda_{n,3})}{\xi_{n,2}}e^{\lambda_{n,2}t}-\frac{(\lambda_{n,1}+\lambda_{n,2})}{\xi_{n,3}}e^{\lambda_{n,3}t}\right)\right]\varphi_{n}(x)\notag\\
&+\sum_{n=1}^{\infty}\left[u_{2,n}\left(\frac{1}{\xi_{n,1}}e^{\lambda_{n,1}t}-\frac{1}{\xi_{n,2}}e^{\lambda_{n,2}t}+\frac{1}{\xi_{n,3}}e^{\lambda_{n,3}t}\right)\right]\varphi_{n}(x).
\end{align}
Letting $A_n(t)$, $B_n(t)$ and $C_{n}(t)$ be given as in \eqref{An}, \eqref{Bn} and \eqref{Cn}, respectively, we obtain that  \eqref{eq-9} follows from  \eqref{9}.  

A tedious but simple calculation gives
\begin{align*}
w(x,0)=\sum_{n=1}^{\infty}\Big(A_{n}(0)u_{0,n}+B_{n}(0)u_{1,n}+C_{n}(0)u_{2,n}\Big)\varphi_{n}(x)=\sum_{n=1}^{\infty}u_{0,n}\varphi_{n}(x)=u_0(x),
\end{align*}
\begin{align*}
w_t(x,0)=\sum_{n=1}^{\infty}\Big(A_{n}'(0)u_{0,n}+B_{n}'(0)u_{1,n}+C_{n}'(0)u_{2,n}\Big)\varphi_{n}(x)=\sum_{n=1}^{\infty}u_{1,n}\varphi_{n}(x)=u_1(x),
\end{align*}
and
\begin{align*}
w_{tt}(x,0)=\sum_{n=1}^{\infty}\Big(A_{n}''(0)u_{0,n}+B_{n}''(0)u_{1,n}+C_{n}''(0)u_{2,n}\Big)\varphi_{n}(x)=\sum_{n=1}^{\infty}u_{2,n}\varphi_{n}(x)=u_2(x).
\end{align*}

It is straightforward to show that $w$ given in \eqref{eq-9} has the regularity \eqref{regu}. Since we are not interested with solutions of \eqref{main-EQ-3-2}, we leave the verification to the interested reader. The proof is finished.
\end{proof}


Next, we consider the non-homogeneous system \eqref{main-EQ-2}.


\begin{theorem}\label{theo-44}
For every $g\in\mathcal D((\Omc)\times (0,T))$, the system \eqref{main-EQ-2} has a unique weak (classical solution) $(v,v_t,v_{tt})$ such that $v\in C^\infty([0,T];W^{s,2}(\RR^N))$ and is given by
\begin{align}\label{eq-22-1}
v(x,t)=\sum_{n=1}^{\infty}\left(\int_0^t \Big( g(\cdot,\tau),\mathcal{N}_{s}\varphi_{n} \Big)_{L^2(\Omc)}\frac{1}{\lambda_{n}}\Big[C_{n}'''(t-\tau)+\alpha C_{n}''(t-\tau)\Big] d\tau\right)\varphi_{n}(x).
\end{align}
Moreover, there is a constant $C>0$ such that for all $t\in [0,T]$ and $m\in\NN_0=\NN\cup\{0\}$, 
\begin{align}\label{ESt-Ind}
\|\partial_t^mv(\cdot,t)\|_{W^{s,2}(\RR^N)}\le C\Big(\|\partial_t^{m+3}g+\alpha\partial_t^{m+2}g\|_{L^\infty((0,T);W^{s,2}(\Omc))}+\|\partial_t^mg(\cdot,t)\|_{W^{s,2}(\Omc)}\Big).
\end{align}
\end{theorem}

\begin{proof}
We proof the theorem in several steps.\\

{\bf Step 1}: Consider the following elliptic Dirichlet exterior problem:

\begin{align}\label{13}
\begin{cases}
 (-\Delta)^{s} \phi =0 & \Omega,\\
\phi=g & \ \Omc.
\end{cases}
\end{align}
We have shown in Proposition \ref{proposi-33} that for every $g\in W^{s,2}(\Omc)$,
there is a unique $\phi\in W^{s,2}(\RR^N)$ solution of \eqref{13}, and there is a constant $C>0$ such that
\begin{align}\label{E-DP-D}
\|\phi\|_{W^{s,2}(\RR^N)}\le C\|g\|_{W^{s,2}(\Omc)}.
\end{align}
Since $g$ depends on $(x,t)$, then $\phi$ also depends on $(x,t)$.  If in \eqref{13} one replaces $g$ by $\partial_t^mg$, $m\in\NN$, then the associated unique solution is given by $\partial_t^m\phi$ for every $m\in\NN_0$. From this, we can deduce that $\phi\in C^\infty([0,T], W^{s,2}(\RR^N))$.

Now let $v$ be a solution of \eqref{main-EQ-2} and set $w:=v-\phi$. Then a simple calculation gives
\begin{align*}
&w_{ttt}+\alpha w_{tt}+c^2(-\Delta)^sw+b(-\Delta)^sw_t\\
=&\;v_{ttt}-\phi_{ttt}+\alpha v_{tt}-\alpha \phi_{tt}+c^2(-\Delta)^sv-c^2(-\Delta)^s\phi+b(-\Delta)^sv_t-b(-\Delta)^s\phi_t\\
=&\;v_{ttt}+\alpha v_{tt}+c^2(-\Delta)^sv+b(-\Delta)^sv_t-\phi_{ttt}-\alpha \phi_{tt}\\
=&-\phi_{ttt}-\alpha\phi_{tt} \;\;\mbox{ in }\;\Omega\times(0,T).
\end{align*}
In addition
\begin{align*}
w=v-\phi=g-g=0\;\mbox{ in }\;(\Omc)\times (0,T),
\end{align*}
and 
\begin{equation*}
\begin{cases}
w(\cdot,0)=v(\cdot,0)-\phi(\cdot,0)=-\phi(\cdot,0)\;\;\;&\mbox{ in }\;\Omega,\\
w_t(\cdot,0)=v_t(\cdot,0)-\phi_t(\cdot,0)=-\phi_t(\cdot,0)\;\;\;\;&\mbox{ in }\;\Omega,\\
w_{tt}(\cdot,0)=v_{tt}(\cdot,0)-\phi_{tt}(\cdot,0)=-\phi_{tt}(\cdot,0)\;\;\;\;&\mbox{ in }\;\Omega.
\end{cases}
\end{equation*}
Since $g\in \mathcal D((\Omc)\times (0,T))$, we have that $\phi(\cdot,0)=\partial_t\phi(\cdot,0)=\partial_{tt}\phi(\cdot,0)=0$ in $\Omega$.  We have shown that a solution $(v,v_t,v_{tt})$ of \eqref{main-EQ-2} can be decomposed as $v=\phi+w$, where  $(w,w_t,w_{tt})$ solves the system
\begin{equation}\label{eq-w-D}
\begin{cases}
w_{ttt}+\alpha w_{tt}+c^2(-\Delta)^sw+b(-\Delta)^sw_t=-\phi_{ttt}-\alpha\phi_{tt}\;\;&\mbox{ in }\;\Omega\times (0,T),\\
w=0  &\mbox{ in }\;(\Omc)\times (0,T),\\
w(\cdot,0)=0,\;w_{t}(\cdot,0)=0,\; w_{tt}(\cdot,0)=0\;\;&\mbox{ in }\;\Omega.
\end{cases}
\end{equation}
We notice that $\phi\in C^\infty([0,T];W^{s,2}(\RR^N))$. \\

{\bf Step 2}: We observe that letting 
\begin{equation*}
W=\left(
\begin{array}{c}
w\\w_t \\ w_{tt}\\
\end{array}
\right)\;\;\mbox{ and }\;   \Phi=\left(
\begin{array}{c}
0\\0\\-\phi_{ttt}-\alpha\phi_{tt}\\ 
\end{array}
\right),
\end{equation*}
then the system \eqref{eq-w-D} can be rewritten as the following first order Cauchy problem
\begin{equation}\label{eq-CP}
\begin{cases}
W_t+\mathcal A W=\Phi\;\;&\mbox{ in }\;\Omega\times(0,T),\\
W(0)=\left(\begin{array}{c}
0\\0\\0\\ 
\end{array}
\right)&\mbox{ in }\;\Omega,
\end{cases}
\end{equation}
where $\mathcal A$ is the matrix operator defined in \eqref{eq-42}. Using \cite[Corollary 1.2]{KLM}, we get that the first order Cauchy problem \eqref{eq-CP} has a unique classical solution $W$ and hence,
\eqref{eq-w-D} has a unique weak (classical) solution $(w,w_t,w_{tt})$ such that $w\in C^\infty([0,T];W^{s,2}(\RR^N))$ and is given by

\begin{align}\label{eq-17}
w(x,t)=-\sum_{n=1}^{\infty}\left(\int_0^t \Big( \phi_{\tau\tau\tau}(\cdot,\tau)+\alpha\phi_{\tau\tau}(\cdot,\tau),\varphi_{n}\Big)_{L^2(\Omega)}C_{n}(t-\tau)d\tau\right)\varphi_{n}(x),
\end{align}
where we recall that $C_n$ is given in \eqref{Cn}.
Integrating  \eqref{eq-17} by parts we get that
\begin{align}\label{eq-18}
w(x,t)=&-\sum_{n=1}^{\infty}\left(\int_0^t ( \phi(\cdot,\tau),\varphi_{n})_{L^2(\Omega)}\Big(C_{n}'''(t-\tau)+\alpha C_{n}''(t-\tau)\Big)d\tau\right)\varphi_{n}(x)\notag\\
&-\sum_{n=1}^{\infty}\left((\phi_{\tau\tau}(\cdot,\tau),\varphi_{n})_{L^2(\Omega)}C_{n}(t-\tau)\Big|_0^t\right)\varphi_{n}(x)-\sum_{n=1}^{\infty}\left((\phi_{\tau}(\cdot,\tau),\varphi_{n})_{L^2(\Omega)}C_{n}'(t-\tau)\Big|_0^t\right)\varphi_{n}(x)\notag\\
&-\sum_{n=1}^{\infty}\left((\phi(\cdot,\tau),\varphi_{n})_{L^2(\Omega)}C_{n}''(t-\tau)\Big|_0^t\right)\varphi_{n}(x)-\alpha\sum_{n=1}^{\infty}\left((\phi_{\tau}(\cdot,\tau),\varphi_{n})_{L^2(\Omega)}C_{n}(t-\tau)\Big|_0^t\right)\varphi_{n}(x)\notag\\
&-\alpha\sum_{n=1}^{\infty}\left((\phi(\cdot,\tau),\varphi_{n})_{L^2(\Omega)}C_{n}'(t-\tau)\Big|_0^t\right)\varphi_{n}(x).
\end{align}
We observe that  $C_{n}(0)=C_{n}'(0)=0$ and $C_{n}''(0)=1$ for all $n\in\NN$. Since $\phi(\cdot,0)=\phi_{t}(\cdot,0)=\phi_{tt}(\cdot,0)=0$, we get that
\begin{align}\label{eq-20}
w(x,t)=-\phi(x,t)-\sum_{n=1}^{\infty}\left(\int_0^t ( \phi(\cdot,\tau),\varphi_{n})_{L^2(\Omega)}\Big(C_{n}'''(t-\tau)+\alpha C_{n}''(t-\tau)\Big)d\tau\right)\varphi_{n}(x).
\end{align}
Using the fact that $\varphi_{n}$ satisfies \eqref{ei-val-pro} and the integration by parts formulas \eqref{Int-Part}-\eqref{Int-Part-2}, we obtain 

\begin{align}\label{eq-21}
\Big(\phi(\cdot,\tau),\lambda_{n}\varphi_{n}\Big)_{L^2(\Omega)}&=\Big( \phi(\cdot,\tau),(-\Delta)^s\varphi_{n}\Big)_{L^2(\Omega)}\nonumber\\
&=\Big((-\Delta)^s \phi(\cdot,\tau),\varphi_{n}\Big)_{L^2(\Omega)}-\int_{\RR^{N}\setminus\Omega}\Big(\phi\mathcal{N}_{s}\varphi_{n}-\varphi_{n}\mathcal{N}_{s}\phi\Big)\;dx\nonumber\\
&=-\int_{\RR^{N}\setminus\Omega}g\mathcal{N}_{s}\varphi_{n}dx.
\end{align}
From \eqref{eq-20} and \eqref{eq-21} we can deduce that
\begin{align}\label{eq-22}
w(x,t)=-\phi(x,t)+\sum_{n=1}^{\infty}\left(\int_0^t \Big( g(\cdot,\tau),\mathcal{N}_{s}\varphi_{n} \Big)_{L^2(\Omc)}\frac{1}{\lambda_{n}}\Big(C_{n}'''(t-\tau)+\alpha C_{n}''(t-\tau)\Big)d\tau\right)\varphi_{n}(x).
\end{align}
We have shown \eqref{eq-22-1}. Since $\phi,w\in C^\infty([0,T];W^{s,2}(\RR^N))$, it follows that $v\in C^\infty([0,T];W^{s,2}(\RR^N))$.\\

{\bf Step 3}: Using \eqref{eq-17} and calculating, we get that for every $t\in [0,T]$,
\begin{align}\label{Int-Es}
\|(-\Delta)_D^sw(\cdot,t)\|_{L^2(\Omega)}=&\left\|\sum_{n=1}^\infty\lambda_n\left(\int_0^t \left( \phi_{\tau\tau\tau}(\cdot,\tau)+\alpha\phi_{\tau\tau}(\cdot,\tau),\varphi_{n}\right)_{L^2(\Omega)}C_{n}(t-\tau)d\tau\right)\varphi_{n}\right\|_{L^2(\Omega)}\notag\\
\le &\int_0^t\left\|\sum_{n=1}^\infty\lambda_n\left( \phi_{\tau\tau\tau}(\cdot,\tau)+\alpha\phi_{\tau\tau}(\cdot,\tau),\varphi_{n}\right)_{L^2(\Omega)}C_{n}(t-\tau)\;d\tau\varphi_{n}\right\|_{L^2(\Omega)}\notag\\
\le&\int_0^t\left(\sum_{n=1}^\infty\Big|\left(\phi_{\tau\tau\tau}(\cdot,\tau)+\alpha\phi_{\tau\tau}(\cdot,\tau),\varphi_{n}\right)_{L^2(\Omega)}\Big|^2\Big|\lambda_nC_{n}(t-\tau)\Big|^2\right)^{\frac 12}\;d\tau.
\end{align}
Using the asymptotic behavior \eqref{asym-eigenvalues}, we obtain (see Lemma \ref{acota} below) that there is a constant $C>0$, independent of $n$, such that
\begin{align}\label{eqB}
\left|\lambda_nC_n(t)\right|\le C,\;\;\forall\;n\in\NN\;\mbox{ and }\; t\in [0,T].
\end{align}
It follows from \eqref{Int-Es}, \eqref{eqB} and \eqref{E-DP-D} that for every $t\in [0,T]$,
\begin{align}\label{Int-Es-2}
\|(-\Delta)_D^sw(\cdot,t)\|_{L^2(\Omega)}\le &C\int_0^t\|\phi_{\tau\tau\tau}(\cdot,\tau)+\alpha\phi_{\tau\tau}(\cdot,\tau)\|_{L^2(\Omega)}\;d\tau\notag\\
\le &C\int_0^t\|\phi_{\tau\tau\tau}(\cdot,\tau)+\alpha\phi_{\tau\tau}(\cdot,\tau)\|_{W^{s,2}(\RR^N)}\;d\tau\notag\\
\le & C\int_0^t\|g_{\tau\tau\tau}(\cdot,\tau)+\alpha g_{\tau\tau}(\cdot,\tau)\|_{W^{s,2}(\Omc)}\;d\tau\notag\\
\le & CT\|g_{ttt}+\alpha g_{tt}\|_{L^\infty((0,T);W^{s,2}(\Omc))}.
\end{align}
Using \eqref{Int-Es-2}, we get that for every $t\in [0,T]$,
\begin{align*}
\|v(\cdot,t)\|_{W^{s,2}(\RR^N)}\le& C\left(\|(-\Delta)_D^sw(\cdot,t)\|_{L^2(\Omega)}+\|\phi(\cdot,t)\|_{W^{s,2}(\RR^N)}\right)\\
\le &C\left(\|g_{ttt}+\alpha g_{tt}\|_{L^\infty((0,T);W^{s,2}(\Omc))}+\|g(\cdot,t)\|_{W^{s,2}(\Omc)}\right).
\end{align*}
We have shown  \eqref{ESt-Ind} for $m=0$.
Proceeding by induction on $m$ we can easily deduce  \eqref{ESt-Ind} for every $m\in\NN_0$. The proof is finished.
\end{proof}

We conclude this subsection with the proof of our main result on existence and uniqueness of weak solutions for the system \eqref{SD-WE}.

\begin{proof}[\bf Proof of Theorem \ref{th:2.3}]
We have shown in Section \ref{sec-2} that a solution $(u,u_t,u_{tt})$ of \eqref{SD-WE} can be decomposed into $u=v+w$ where $(v,v_t,v_{tt})$ solves \eqref{main-EQ-2} and $(w,w_t,w_{tt})$ is a solution of \eqref{main-EQ-3}. Now the result follows from  Proposition \ref{pro-sol-SG} and Theorem \ref{theo-44}.
\end{proof}

\subsection{Series solutions of the dual system}

Now we consider the dual system \eqref{ACP-Dual}. 
Let
\begin{align*}
\psi_{0,n}:=(\psi_0,\varphi_n)_{L^2(\Omega)},\quad \psi_{1,n}:=(\psi_1,\varphi_n)_{L^2(\Omega)},\quad\mbox{and}\quad \psi_{2,n}:=(\psi_2,\varphi_n)_{L^2(\Omega)}.
\end{align*}
Throughout this subsection we will denote
$D_{n}(t)=A_{n}(t)$, $E_{n}(t)=-B_{n}(t)$ and $F_{n}(t)=C_{n}(t)$, where $A_n(t)$, $B_n(t)$ and $C_n(t)$ are given in \eqref{An}, \eqref{Bn} and \eqref{Cn}, respectively. We begin with the following technical Lemma that is crucial in the proof of Theorem \ref{theo-48}.

\begin{lemma}\label{acota}
There is a constant $C>0$ (independent of $n$) such that for every $t\in [0,T]$, 
\begin{align}\label{ES-D}
\max\left\{\left|D_n(t)\right|^2 ,|D_n'(t)|^2, \left|\frac{D_n''(t)}{\lambda_{n}^{\frac 12}}\right|^2,\left|\frac{D_n''(t)}{\lambda_{n}^{\frac 32}}\right|^2\right\}\le C,
\end{align}
\begin{align}\label{ES-E}
\max\left\{\left|E_n(t)\right|^2, |E_n'(t)|^2, \left|\frac{E_n''(t)}{\lambda_{n}^{\frac 12}}\right|^2,\left|\frac{E_n''(t)}{\lambda_{n}^{\frac 32}}\right|^2\right\}\le C,
\end{align}
and
\begin{align}\label{ES-F}
\max\left\{\left|\lambda_{n}^{\frac 12}F_n(t)\right|^2 , \left|\lambda_{n}F_n(t)\right|^2, |\lambda_n^{\frac 12}F_n'(t)|^2, |F_n''(t)|^2, \left|\frac{F_n''(t)}{\lambda_{n}^{\frac 12}}\right|^2\right\}\le C.
\end{align}
\end{lemma}

\begin{proof} 
We rewrite the functions $D_{n}(t)$, $E_{n}(t)$ and $F_{n}(t)$ as follows:
\begin{align}
D_n(t)=\sum_{j=1}^{3}D_{n,j}e^{\lambda_{n,j}t},\quad E_n(t)=\sum_{j=1}^{3}E_{n,j}e^{\lambda_{n,j}t}, \quad F_n(t)=\sum_{j=1}^{3}F_{n,j}e^{\lambda_{n,j}t},
\end{align}
where
\begin{align*}
& D_{n,1}=\frac{\lambda_{n,2}\lambda_{n,3}}{\xi_{n,1}},\quad D_{n,2}=-\frac{\lambda_{n,1}\lambda_{n,3}}{\xi_{n,2}},\quad D_{n,3}=\frac{\lambda_{n,1}\lambda_{n,2}}{\xi_{n,3}},\\
& E_{n,1}=\frac{\lambda_{n,2}+\lambda_{n,3}}{\xi_{n,1}},\quad E_{n,2}=-\frac{\lambda_{n,1}+\lambda_{n,3}}{\xi_{n,2}},\quad E_{n,3}=\frac{\lambda_{n,1}+\lambda_{n,2}}{\xi_{n,3}},\\
& F_{n,1}=\frac{1}{\xi_{n,1}},\quad F_{n,2}=-\frac{1}{\xi_{n,2}},\quad F_{n,3}=\frac{1}{\xi_{n,3}}.
\end{align*}

We proof the Lemma in several steps. First of all, we notice that sine $\lambda_{n,1}<0$ and $\operatorname{Re}(\lambda_{n,j})<0$, for $j=2,3$, it follows that $\left|e^{\lambda_{n,jt}}\right|\le 1$ for $j=1,2,3$.\\

{\bf Step 1}: Observe that
\begin{align}\label{lambda1}
\left|\frac{\lambda_n^{\frac 12}}{\lambda_{n,2}}\right|\leq \frac{1}{\left|\frac{|\text{Re}(\lambda_{n,2})|}{\lambda_n^{\frac 12}}-\frac{|\text{Im}(\lambda_{n,2})|}{\lambda_n^{\frac 12}}\right|}.
\end{align}
Since $\text{Re}(\lambda_{n,2})\searrow -\frac{\gamma}{2}$, $|\text{Im}(\lambda_{n,2})|\sim \sqrt{b\lambda_{n}}$ and $\lambda_{n}\to +\infty$ as $n\to\infty$, it follows that the sequence $\left\{\left|\frac{\lambda_n^{\frac 12}}{\lambda_{n,2}}\right|\right\}_{n\in\NN}$ is convergent. Using that $\lambda_{n,3}=\overline{\lambda_{n,2}}$ we also obtain that the sequence $\left\{\left|\frac{\lambda_n^{\frac 12}}{\lambda_{n,3}}\right|\right\}_{n\in\NN}$ is convergent. Thus these two sequences are bounded.\\

{\bf Step 2}: We claim that the sequence $\{|\lambda_{n,j}D_{n,j}|\}_{n\in\NN}$ is convergent. Indeed, since $\lambda_{n,3}=\overline{\lambda_{n,2}}$, it suffices prove the cases $j=1,2$. We observe that

\begin{align*}
|\lambda_{n,1}D_{n,1}|&=\left|\lambda_{n,1}\frac{\lambda_{n,2}\overline{\lambda_{n,2}}}{(\lambda_{n,1}-\lambda_{n,2})(\lambda_{n,1}-\overline{\lambda_{n,2}})}\right|=|\lambda_{n,1}|\left|\frac{\lambda_{n,2}}{\lambda_{n,1}-\lambda_{n,2}}\right|^2\\
&\leq |\lambda_{n,1}|\left|\frac{\frac{|\text{Re}(\lambda_{n,2})|}{|\text{Im}(\lambda_{n,2})|}+1}{\left|\frac{|\text{Re}(\lambda_{n,1}-\lambda_{n,2})|}{|\text{Im}(\lambda_{n,2})|}-\frac{|\text{Im}(\lambda_{n,1}-\lambda_{n,2})|}{|\text{Im}(\lambda_{n,2})|}\right|}\right|^2,
\end{align*}
and
\begin{align*}
|\lambda_{n,2}D_{n,2}|&=\left|\lambda_{n,1}\frac{\lambda_{n,2}\overline{\lambda_{n,2}}}{(\lambda_{n,1}-\lambda_{n,2})(\lambda_{n,2}-\overline{\lambda_{n,2}})}\right|=|\lambda_{n,1}|\frac{|\lambda_{n,2}|^2}{|\lambda_{n,1}-\lambda_{n,2}||2i\text{Im}(\lambda_{n,2})|}\\
&\leq |\lambda_{n,1}|\frac{\frac{|\text{Re}(\lambda_{n,2})|}{|\text{Im}(\lambda_{n,2})|}+1}{\left|\frac{|\text{Re}(\lambda_{n,1}-\lambda_{n,2})|}{|\text{Im}(\lambda_{n,2})|}-\frac{|\text{Im}(\lambda_{n,1}-\lambda_{n,2})|}{|\text{Im}(\lambda_{n,2})|}\right|}\frac{\frac{|\text{Re}(\lambda_{n,2})|}{|\text{Im}(\lambda_{n,2})|}+1}{\frac{2|\text{Im}(\lambda_{n,2})|}{|\text{Im}(\lambda_{n,2})|}}.
\end{align*}
From \eqref{asym-eigenvalues}, we can deduce the convergence of the sequence $\{|\lambda_{n,j}D_{n,j}|\}_{n\in\NN}$.\\

{\bf Step 3}: Now we prove \eqref{ES-D}.  Notice that 
\begin{align*}
|D_{n,j}|=\left|\frac{1}{\lambda_{n,j}}\right| |\lambda_{n,j}D_{n,j}|,
\end{align*}
and since $\{|\lambda_{n,j} D_{n,j}|\}_{n\in\NN}$ and $\{\left|\frac{1}{\lambda_{n,j}}\right|\}_{n\in\NN}$ are convergent, we obtain that $\{|D_{n,j}|\}_{n\in\NN}$, for $j=1,2$, is convergent. Thus, the sequence $\{|D_{n}(t)|\}_{n\in\NN}$ is bounded.

The case of  $\{|D_n'(t)|\}_{n\in\NN}$ is a simple consequence of Step 2 due to the fact that
\begin{align*}
D_{n}'(t)=\sum_{j=1}^{3}\lambda_{n,j}D_{n,j}e^{\lambda_{n,j}t}.
\end{align*}

For the convergence of the sequence $\left\{\left|\frac{D_n''(t)}{\lambda_{n}^{\frac 12}}\right|\right\}_{n\in\NN}$, we observe that
\begin{align*}
\frac{D_n''(t)}{\lambda_n^{\frac 12}}=\sum_{j=1}^{3}\frac{\lambda_{n,j}^2}{\lambda_n^{\frac 12}}D_{n,j}e^{\lambda_{n,j}t}=\sum_{j=1}^{3}\frac{\lambda_{n,j}}{\lambda_n^{\frac 12}} \lambda_{n,j}D_{n,j}e^{\lambda_{n,j}t}.
\end{align*}
Since $\{|\lambda_{n,j} D_{n,j}|\}_{n\in\NN}$ and $\left\{\left|\D\frac{\lambda_n^{\frac 12}}{\lambda_{n,2}}\right|\right\}_{n\in\NN}$ are convergent, we obtain that $\left\{\left|\frac{D_n''(t)}{\lambda_{n}^{\frac 12}}\right|\right\}_{n\in\NN}$ is bounded. From the above computation we also deduce that the sequence $\left\{\left|\frac{D_n''(t)}{\lambda_{n}^{\frac 32}}\right|\right\}_{n\in\NN}$ is  bounded. Therefore, we can deduce that \eqref{ES-D} holds.\\

{\bf Step 4}: To prove \eqref{ES-E}, we observe the following:
\begin{align*}
E_{n,1}=-\left(\frac{1}{\lambda_{n,3}}+\frac{1}{\lambda_{n,2}}\right)D_{n,1},\quad 
E_{n,2}=\left(\frac{1}{\lambda_{n,3}}+\frac{1}{\lambda_{n,1}}\right)D_{n,2},\\
E_{n,3}=-\left(\frac{1}{\lambda_{n,2}}+\frac{1}{\lambda_{n,1}}\right)D_{n,3},
\end{align*}
\begin{align*}
\lambda_{n,1}E_{n,1}=-\left(\frac{1}{\lambda_{n,3}}+\frac{1}{\lambda_{n,2}}\right)\lambda_{n,1}D_{n,1},\quad 
\lambda_{n,2}E_{n,2}=\left(\frac{1}{\lambda_{n,3}}+\frac{1}{\lambda_{n,1}}\right)\lambda_{n,2}D_{n,2},\\
\lambda_{n,3}E_{n,3}=-\left(\frac{1}{\lambda_{n,2}}+\frac{1}{\lambda_{n,1}}\right)\lambda_{n,3}D_{n,3},
\end{align*}
and
\begin{align*}
\frac{\lambda_{n,1}^2}{\lambda_n^{\frac 12}}E_{n,1}=-\left(\frac{1}{\lambda_{n,3}}+\frac{1}{\lambda_{n,2}}\right)\frac{\lambda_{n,1}}{\lambda_n^{\frac 12}}\lambda_{n,1}D_{n,1},\quad 
\frac{\lambda_{n,2}^2}{\lambda_n^{\frac 12}}E_{n,2}=\left(\frac{1}{\lambda_{n,3}}+\frac{1}{\lambda_{n,1}}\right)\frac{\lambda_{n,2}}{\lambda_n^{\frac 12}}\lambda_{n,2}D_{n,2},\\
\frac{\lambda_{n,3}^2}{\lambda_n^{\frac 12}}E_{n,3}=-\left(\frac{1}{\lambda_{n,2}}+\frac{1}{\lambda_{n,1}}\right)\frac{\lambda_{n,3}}{\lambda_n^{\frac 12}}\lambda_{n,3}D_{n,3}.
\end{align*}
Then, using the previous steps  we obtain \eqref{ES-E}.\\

{\bf Step 5}: We claim that $\{|\lambda_n F_{n,j}|\}_{n\in\NN}$ is convergent. Indeed,
\begin{align*}
|\lambda_n F_{n,j}|&\leq \left|\frac{\lambda_n}{\xi_{n,1}}\right|+\left|\frac{\lambda_n}{\xi_{n,2}}\right|+\left|\frac{\lambda_n}{\xi_{n,3}}\right|\\
&= \left|\frac{\lambda_n}{\lambda_{n,2}\lambda_{n,3}}\right| \left|\frac{\lambda_{n,2}\lambda_{n,3}}{\xi_{n,1}}\right|+\left|\frac{\lambda_n}{\lambda_{n,1}\lambda_{n,2}\lambda_{n,3}}\right| \left|\frac{\lambda_{n,1}\lambda_{n,2}\lambda_{n,3}}{\xi_{n,2}}\right|+\left|\frac{\lambda_n}{\lambda_{n,1}\lambda_{n,2}\lambda_{n,3}}\right| \left|\frac{\lambda_{n,1}\lambda_{n,2}\lambda_{n,3}}{\xi_{n,3}}\right|\\
&= \left|\frac{\lambda_n}{\lambda_{n,2}\lambda_{n,3}}\right||D_{n,1}|+\left|\frac{\lambda_n}{\lambda_{n,1}\lambda_{n,2}\lambda_{n,3}}\right||\lambda_{n,2}D_{n,2}|+\left|\frac{\lambda_n}{\lambda_{n,1}\lambda_{n,2}\lambda_{n,3}}\right||\lambda_{n,3}D_{n,3}|.
\end{align*}
We observe that
\begin{align*}
\left|\frac{\lambda_n}{\lambda_{n,1}\lambda_{n,2}\lambda_{n,3}}\right| =\frac{1}{|\lambda_{n,1}|}\left|\frac{\lambda_n}{\lambda_{n,2}\overline{\lambda_{n,2}}}\right|
=\frac{1}{|\lambda_{n,1}|}\left|\frac{\lambda_n}{|\lambda_{n,2}|^2}\right|
=\frac{1}{|\lambda_{n,1}|}\frac{1}{\left|\frac{|\text{Re}(\lambda_{n,2})|^2}{\lambda_n}+\frac{|\text{Im}(\lambda_{n,2})|^2}{\lambda_n}\right|}.
\end{align*}
Using the convergence property \eqref{asym-eigenvalues}, we obtain the desired result.\\

{\bf Step 6}: Finally, we  prove \eqref{ES-F}.  Observe that
\begin{align*}
\lambda_n^{\frac 12}F_n(t)&=\frac{\lambda_n^{\frac 12}}{\lambda_{n,2}\lambda_{n,3}}\frac{\lambda_{n,2}\lambda_{n,3}}{\xi_{n,1}}e^{\lambda_{n,1}t}-\frac{\lambda_n^{\frac 12}}{\lambda_{n,1}\lambda_{n,3}}\frac{\lambda_{n,1}\lambda_{n,3}}{\xi_{n,2}}e^{\lambda_{n,2}t}+\frac{\lambda_n^{\frac 12}}{\lambda_{n,1}\lambda_{n,2}}\frac{\lambda_{n,1}\lambda_{n,2}}{\xi_{n,3}}e^{\lambda_{n,3}t}\\
&=\frac{\lambda_n^{\frac 12}}{\lambda_{n,2}\lambda_{n,3}}D_{n,1}e^{\lambda_{n,1}t}-\frac{\lambda_n^{\frac 12}}{\lambda_{n,1}\lambda_{n,3}}D_{n,2}e^{\lambda_{n,2}t}+\frac{\lambda_n^{\frac 12}}{\lambda_{n,1}\lambda_{n,2}}D_{n,3}e^{\lambda_{n,3}t}.
\end{align*}
Therefore, it is easy to see that $\{|\lambda_n^{\frac 12}F_n(t)|\}_n$ is a bounded sequence. 

With a similar argument, we can  deduce that the following sequences $\left\{|\lambda_n^{\frac 12}F_n'(t)|\right\}_{n\in\NN}$, $\left\{|F_n''(t)|\right\}_{n\in\NN}$ and $\left\{\left|\frac{F_n''(t)}{\lambda_n^{\frac 12}}\right|\right\}_{n\in\NN}$ are bounded.
The proof is finished.
\end{proof}

\begin{proof}[\bf Proof of Theorem \ref{theo-48}]
Let
\begin{align}\label{24.1}
\psi_0=\sum_{n=1}^{\infty}\psi_{0,n}\varphi_{n}, \quad \psi_1=\sum_{n=1}^{\infty}\psi_{1,n}\varphi_{n}, \quad \psi_2=\sum_{n=1}^{\infty}\psi_{2,n}\varphi_{n}.
\end{align}
We proof the theorem in several steps. Here we include more details.\\

{\bf Step 1}: Proceeding in the same way as the proof of Proposition \ref{pro-sol-SG}, we easily get that 
\begin{align}\label{25}
\psi(x,t)=\sum_{n=1}^{\infty}\Big[D_{n}(T-t)\psi_{0,n}+E_{n}(T-t)\psi_{1,n}+F_{n}(T-t)\psi_{2,n}\Big]\varphi_{n}(x),
\end{align}
where $D_{n}(t)=A_{n}(t)$, $E_{n}(t)=-B_{n}(t)$ and $F_n(t)=C_n(t)$. In addition, a simple calculation gives $\psi(x,T)=\psi_0(x)$, $\psi_t(x,T)=-\psi_1(x)$ and $\psi_{tt}(x,T)=\psi_2(x)$ for a.e. $x\in\Omega$.

Let us show that $\psi$ satisfies the regularity and variational identity requirements. Let $1\le n\le m$ and set
\begin{align*}
\psi_m(x,t)=\sum_{n=1}^{m}\Big[D_{n}(T-t)\psi_{0,n}+E_{n}(T-t)\psi_{1,n}+F_{n}(T-t)\psi_{2,n}\Big]\varphi_{n}(x).
\end{align*}
For every $m,\tilde m\in\NN$ with $m>\tilde m$ and $t\in [0,T]$, we have that
\begin{align}\label{Nor1}
\|\psi_m(x,t)-\psi_{\tilde m}(x,t)\|_{W_0^{s,2}(\bOm)}^2=&\sum_{n=\tilde m+1}^m\Big|\lambda_n^{\frac 12}D_{n}(T-t)\psi_{0,n}+\lambda_n^{\frac 12}E_{n}(T-t)\psi_{1,n}+\lambda_n^{\frac 12}F_{n}(T-t)\psi_{2,n}\Big|^2\notag\\
\le &2\sum_{n=\tilde m+1}^m\Big|\lambda_n^{\frac 12}D_{n}(T-t)\psi_{0,n}\Big|^2+2\sum_{n=\tilde m+1}^m\Big|\lambda_n^{\frac 12}E_{n}(T-t)\psi_{1,n}\Big|^2\notag\\
&+2\sum_{n=\tilde m+1}^m\Big|\lambda_n^{\frac 12}F_{n}(T-t)\psi_{2,n}\Big|^2.
\end{align}
Using \eqref{ES-D}, \eqref{ES-E} and \eqref{ES-F} we get from \eqref{Nor1} that for every $m,\tilde m\in\NN$ with $m>\tilde m$ and $t\in [0,T]$,
\begin{multline*}
\|\psi_m(x,t)-\psi_{\tilde m}(x,t)\|_{W_0^{s,2}(\bOm)}^2\le  C\left(\sum_{n=\tilde m+1}^m\Big|\lambda_n^{\frac 12}\psi_{0,n}\Big|^2+\sum_{n=\tilde m+1}^m\Big|\lambda_n^{\frac 12}\psi_{1,n}\Big|^2 +\sum_{n=\tilde m+1}^m\Big|\psi_{2,n}\Big|^2\right)\longrightarrow 0,
\end{multline*}
 as $m$, $\tilde m\to\infty$.
We have shown that the series
\begin{align*}
\sum_{n=1}^{\infty}\Big[D_{n}(T-t)\psi_{0,n}+E_{n}(T-t)\psi_{1,n}+F_{n}(T-t)\psi_{2,n}\Big]\varphi_{n}\longrightarrow \psi(\cdot,t)\;\mbox{ in }\; W_0^{s,2}(\bOm),
\end{align*}
 and that the convergence is uniform in $t\in [0,T]$. Hence, $\psi\in C([0,T];W_0^{s,2}(\bOm))$. Using \eqref{ES-D}, \eqref{ES-E} and \eqref{ES-F} again we get that there is a constant $C>0$ such that for every $t\in [0,T]$,
 \begin{align}\label{A1}
 \|\psi(\cdot,t)\|_{W_0^{s,2}(\bOm)}\le C\Big(\|\psi_0\|_{W_0^{s,2}(\bOm)}+\|\psi_1\|_{W_0^{s,2}(\bOm)}+\|\psi_2\|_{L^2(\Omega)}\Big).
 \end{align}

 {\bf Step 2}: Next, we show that $\psi_t\in C([0,T];W_0^{s,2}(\bOm))$. Indeed, we have
  \begin{align*}
 (\psi_m)_t(x,t)=-\sum_{n=1}^{m}\Big[D_{n}'(T-t)\psi_{0,n}+E_{n}'(T-t)\psi_{1,n}+F_{n}'(T-t)\psi_{2,n}\Big]\varphi_{n}(x).
 \end{align*}
 Proceeding as above, we obtain that the series
  \begin{align*}
 \sum_{n=1}^{\infty}\Big[D_{n}'(T-t)\psi_{0,n}+E_{n}'(T-t)\psi_{1,n}+F_{n}'(T-t)\psi_{2,n}\Big]\varphi_{n}\longrightarrow \psi_t(\cdot,t)\;\mbox{ in }\;W_0^{s,2}(\bOm),
 \end{align*}
  and the convergence is uniform in $t\in [0,T]$. As in the previous case, using \eqref{ES-D}, \eqref{ES-E} and \eqref{ES-F}, we get that there is a constant $C>0$ such that for every $t\in [0,T]$,
 \begin{align}\label{A2}
\|\psi_t(\cdot,t)\|_{L^2(\Omega)}^2\le C\Big(\|\psi_0\|_{W_0^{s,2}(\bOm)}^2+\|\psi_1\|_{W_0^{s,2}(\bOm)}^2+\|\psi_2\|_{L^2(\Omega)}^2\Big).
\end{align}

{\bf Step 3}:  Next, we claim that $\psi_{tt}\in C([0,T];L^2(\Omega))$. Calculating, we get that 
 \begin{align*}
 (\psi_m)_{tt}(x,t)=\sum_{n=1}^{m}\Big[D_{n}''(T-t)\psi_{0,n}+E_{n}''(T-t)\psi_{1,n}+F_{n}''(T-t)\psi_{2,n}\Big]\varphi_{n}(x).
 \end{align*}
As in Step 1,  we obtain that for every $m,\tilde m\in\NN$ with $m>\tilde m$ and $t\in [0,T]$
\begin{align}\label{Nor2}
&\|\partial_{tt}\psi_m(x,t)-\partial_{tt}\psi_{\tilde m}(x,t)\|_{L^2(\Omega)}^2 \\
&\le 2\sum_{n=\tilde m+1}^m\Big|\lambda_n^{\frac 12}\frac{D_{n}''(T-t)}{\lambda_n^{\frac 12}}\psi_{0,n}\Big|^2+2\sum_{n=\tilde m+1}^m\Big|\lambda_n^{\frac 12}\frac{E_{n}''(T-t)}{\lambda_n^{\frac 12}}\psi_{1,n}\Big|^2+2\sum_{n=\tilde m+1}^m\Big|F_{n}''(T-t)\psi_{2,n}\Big|^2\notag\\
&\le  C\left(\sum_{n=\tilde m+1}^m\Big|\lambda_n^{\frac 12}\psi_{0,n}\Big|^2+\sum_{n=\tilde m+1}^m\Big|\lambda_n^{\frac 12}\psi_{1,n}\Big|^2+\sum_{n=\tilde m+1}^m\Big|\psi_{2,n}\Big|^2\right)\notag\\
\quad&\longrightarrow 0 \;\mbox{ as }\; m,\tilde m\to\infty.
\end{align}

 Again, we can easily deduce that the series
 \begin{align*}
 \sum_{n=1}^{\infty}\Big[D_{n}''(T-t)\psi_{0,n}+E_{n}''(T-t)\psi_{1,n}+F_{n}''(T-t)\psi_{2,n}\Big]\varphi_{n}\longrightarrow \psi_{tt}(\cdot,t)\;\mbox{ in }\;L^2(\Omega),
 \end{align*}
 and the convergence is uniform in $t\in [0,T]$. In addition using \eqref{ES-D}, \eqref{ES-E} and \eqref{ES-F}, we get that there is a constant $C>0$ such that for every $t\in [0,T]$,
\begin{align}\label{A3}
\|\psi_{tt}(\cdot,t)\|_{L^2(\Omega)}^2\le C\Big(\|\psi_0\|_{W_0^{s,2}(\bOm)}^2+\|\psi_1\|_{W_0^{s,2}(\bOm)}^2+\|\psi_2\|_{L^2(\Omega)}^2\Big).
\end{align}

The estimate \eqref{Dual-EST-1} follows from \eqref{A1}, \eqref{A2} and \eqref{A3}.\\

{\bf Step 4}: We show that $\psi_{ttt}\in C([0,T);W^{-s,2}(\bOm))$. Using \eqref{norm-V-2}, \eqref{ES-D}, \eqref{ES-E} and \eqref{ES-F}, we get that for every $t\in [0,T]$,
\begin{align}\label{MW1}
&\|(-\Delta)_D^s\psi(\cdot,t)\|_{W^{-s,2}(\bOm)}^2\notag\\
\le &2\sum_{n=1}^\infty\left(\Big|\lambda_n^{-\frac 12}\lambda_nD_n(T-s)\psi_{0,n}\Big|^2+\Big|\lambda_n^{-\frac 12}\lambda_nE_n(T-s)\psi_{1,n}\Big|^2+\Big|\lambda_n^{-\frac 12}\lambda_nF_n(T-s)\psi_{2,n}\Big|^2\right)\notag\\
\le &2\sum_{n=1}^\infty\left(\Big|\lambda_n^{\frac 12}D_n(T-s)\psi_{0,n}\Big|^2+\Big|\lambda_n^{\frac 12}E_n(T-s)\psi_{1,n}\Big|^2+\Big|\lambda_n^{\frac 12}F_n(T-s)\psi_{2,n}\Big|^2\right)\notag\\
\le &C\Big(\|\psi_0\|_{W_0^{s,2}(\bOm)}^2+\|\psi_1\|_{W_0^{s,2}(\bOm)}^2+\|\psi_2\|_{L^2(\Omega)}^2\Big).
\end{align}
Using \eqref{norm-V-2}, \eqref{ES-D}, \eqref{ES-E} and \eqref{ES-F} again we get that there is a constant $C>0$ such that for every $t\in [0,T]$, 
\begin{align}\label{MW2}
&\|(-\Delta)_D^s\psi_t(\cdot,t)\|_{W^{-s,2}(\bOm)}^2\notag\\
\le &2\sum_{n=1}^\infty\left(\Big|\lambda_n^{-\frac 12}\lambda_nD_n'(T-s)\psi_{0,n}\Big|^2+\Big|\lambda_n^{-\frac 12}\lambda_nE_n'(T-s)\psi_{1,n}\Big|^2+\Big|\lambda_n^{-\frac 12}\lambda_nF_n'(T-s)\psi_{2,n}\Big|^2\right)\notag\\
\le &2\sum_{n=1}^\infty\left(\Big|\lambda_n^{\frac 12}D_n'(T-s)\psi_{0,n}\Big|^2+\Big|\lambda_n^{\frac 12}E_n'(T-s)\psi_{1,n}\Big|^2+\Big|\lambda_n^{\frac 12}F_n'(T-s)\psi_{2,n}\Big|^2\right)\notag\\
\le &C\Big(\|\psi_0\|_{W_0^{s,2}(\bOm)}^2+\|\psi_1\|_{W_0^{s,2}(\bOm)}^2+\|\psi_2\|_{L^2(\Omega)}^2\Big).
\end{align}
Finally, using \eqref{norm-V-2}, \eqref{ES-D}, \eqref{ES-E} and \eqref{ES-F} again we get 
\begin{align}\label{MW3}
\|\psi_{tt}(\cdot,t)\|_{W^{-s,2}(\bOm)}^2\le &2\sum_{n=1}^\infty\left(\Big|\lambda_n^{-\frac 12}D_n''(T-s)\psi_{0,n}\Big|^2+\Big|\lambda_n^{-\frac 12}E_n''(T-s)\psi_{1,n}\Big|^2+\Big|\lambda_n^{-\frac 12}F_n''(T-s)\psi_{2,n}\Big|^2\right)\notag\\
\le &2\sum_{n=1}^\infty\left(\Big|\lambda_n^{\frac 12}\frac{D_n''(T-s)}{\lambda_{n}^{\frac 32}}\psi_{0,n}\Big|^2+\Big|\lambda_n^{\frac 12}\frac{E_n''(T-s)}{\lambda_{n}^{\frac 32}}\psi_{1,n}\Big|^2+\Big|\lambda_n^{-\frac 12}F_n''(T-s)\psi_{2,n}\Big|^2\right)\notag\\
\le &C\Big(\|\psi_0\|_{W_0^{s,2}(\bOm)}^2+\|\psi_1\|_{W_0^{s,2}(\bOm)}^2+\|\psi_2\|_{L^2(\Omega)}^2\Big).
\end{align}

Since $\psi_{ttt}(\cdot,t)=-\alpha\psi_{tt}(\cdot,t)-c^2(-\Delta)_D^s\psi(\cdot,t)+b(-\Delta)_D^s\psi_t(\cdot,t)$, it follows from \eqref{MW1}, \eqref{MW2} and \eqref{MW3} that
\begin{align*}
\|\psi_{ttt}(\cdot,t)\|_{W^{-s,2}(\bOm)}^2\le C\Big(\|\psi_0\|_{W_0^{s,2}(\bOm)}^2+\|\psi_1\|_{W_0^{s,2}(\bOm)}^2+\|\psi_2\|_{L^2(\Omega)}^2\Big),
\end{align*}
 and we have also shown \eqref{Dual-EST-1-2}.  We can also easily deduce that  $\psi_{ttt}\in C([0,T);W^{-s,2}(\bOm))$.\\

{\bf Step 5}: We claim that $\psi\in L^\infty((0,T); D((-\Delta)_D^s))$. It follows from the estimate \eqref{Dual-EST-1} that $\psi\in L^\infty((0,T);L^2(\Omega))$. Since $D((-\Delta)_D^s)\times D((-\Delta)_D^s)\times L^2(\Om)$ is dense in the Banach space $W_0^{s,2}(\bOm)\times  W_0^{s,2}(\bOm)\times  L^2(\Omega)$, it suffices to consider  $(\psi_0,\psi_1,\psi_2)\in D((-\Delta)_D^s)\times D((-\Delta)_D^s)\times L^2(\Om)$. Proceeding as above we get that
\begin{align}\label{M-W}
\|\psi(\cdot,t)\|_{D((-\Delta)_D^s)}^2=&\|(-\Delta)_D^s\psi(\cdot,t)\|_{L^2(\Omega)}^2\notag\\
\le &2\sum_{n=1}^\infty\left(\Big|D_n(T-t)\lambda_n\psi_{0,n}\Big|^2+\Big|E_n(T-t)\lambda_n\psi_{1,n}\Big|^2+\Big|\lambda_nF_n(T-t)\psi_{2,n}\Big|^2\right).
\end{align} 
It follows from \eqref{M-W}, \eqref{ES-D}, \eqref{ES-E} and \eqref{ES-F} that
\begin{align*}
\|\psi(\cdot,t)\|_{D((-\Delta)_D^s)}^2\le C\Big(\|\psi_0\|_{D((-\Delta)_D^s)}^2+\|\psi_1\|_{D((-\Delta)_D^s)}^2+\|\psi_2\|_{L^2(\Omega)}^2\Big).
\end{align*}
Thus $\psi\in L^\infty((0,T); D((-\Delta)_D^s))$ and we have shown the claim.\\

{\bf Step 6}: It is easy to see that the mapping $[0,T)\ni t\to \psi(\cdot,t)\in L^2(\Omc)$ can be analytically extended to $\Sigma_T$. We also recall that for every $t\in [0,T)$ fixed, we have that $\psi(\cdot,t)\in D((-\Delta)_D^s)\subset W^{s,2}(\RR^N)$. Therefore, $\mathcal N_sv(\cdot,t)$ exists and belongs to $L^2(\Omc)$.

We claim that 

\begin{align}\label{28}
\mathcal{N}_{s}\psi(x,t)=\sum_{n=1}^{\infty}\Big(D_{n}(T-t)\psi_{0,n}+E_{n}(T-t)\psi_{1,n}+F_{n}(T-t)\psi_{2,n}\Big)\mathcal{N}_{s}\varphi_{n}(x),
\end{align}
and the series is convergent in $L^2(\RR^{N}\setminus\Omega)$ and that the convergence is uniform in $t\in[0,T)$. Indeed, let $\eta>0$ be fixed but arbitrary and let $t\in[0,T-\eta]$. Let $n,m\in\NN$ with $n>m$. Since $\mathcal{N}_{s}:W^{s,2}(\RR^{N})\to L^2(\RR^{N}\setminus\Omega)$ is bounded, then using \eqref{ES-D}, \eqref{ES-E} and \eqref{ES-F}, we get that there is a constant $C>0$ such that

\begin{align*}
&\left\|\sum_{n=m+1}^{\infty}\Big(D_{n}(T-t)\psi_{0,n}+E_{n}(T-t)\psi_{1,n}+F_{n}(T-t)\psi_{2,n}\Big)\mathcal{N}_{s}\varphi_{n}\right\|_{L^2(\RR^{N}\setminus\Omega)}^2\\
\le &C\left\|\sum_{n=m+1}^{\infty}\Big(D_{n}(T-t)\psi_{0,n}+E_{n}(T-t)\psi_{1,n}+F_{n}(T-t)\psi_{2,n}\Big)\varphi_{n}\right\|_{W_0^{s,2}(\bOm)}^2\\
\leq &C\left(\sum_{n=m+1}^{\infty}|\psi_{0,n}|^2+\sum_{n=m+1}^{\infty}|\psi_{1,n}|^2+\sum_{n=m+1}^{\infty}|\psi_{2,n}|^2\right)\longrightarrow 0\text{ as }m\to \infty.
\end{align*}
Thus, $\mathcal{N}_{s}$ is given by \eqref{28} and the series is convergent in $L^2(\RR^{N}\setminus\Omega)$ uniformly in any compact subset of $[0,T)$.

Besides, we obtain the following continuous dependence on the data for the nonlocal normal derivative. Let $m\in\NN$  and consider
\begin{align}\label{27.1}
\psi_{m}(x,t)=\sum_{n=1}^{m}\Big(D_{n}(T-t)\psi_{0,n}+E_{n}(T-t)\psi_{1,n}+F_{n}(T-t)\psi_{2,n}\Big)\mathcal{N}_{s}\varphi_{n}(x).
\end{align}
Using the fact that the operator $\mathcal{N}_{s}:  W^{s,2}_0(\bOm)\longrightarrow L^2(\RR^{N}\setminus\Omega)$ is bounded,  the continuous embedding $W_0^{s,2}(\bOm)\hookrightarrow L^{2}(\Omega)$, \eqref{ES-D} and \eqref{ES-E} and \eqref{ES-D}, we get that there is a constant $C>0$ such that for every $t\in [0,T]$,

\begin{align}
&\left\|\sum_{k=1}^{m} D_{k}(T-t)\psi_{0,k}\mathcal{N}_{s}\varphi_{k}\right\|_{L^2(\RR^{N}\setminus\Omega)}^2+\left\|\sum_{k=1}^{m} E_{k}(T-t)\psi_{1,k}\mathcal{N}_{s}\varphi_{k}\right\|_{L^2(\RR^{N}\setminus\Omega)}^2+\left\|\sum_{k=1}^{m} F_{k}(T-t)\psi_{2,k}\mathcal{N}_{s}\varphi_{k}\right\|_{L^2(\RR^{N}\setminus\Omega)}^2\notag\\ 
&\leq C\left(\left\|\sum_{k=1}^{m}D_{k}(T-t)\psi_{0,k}\varphi_{k}\right\|_{W_0^{s,2}(\bOm)}^2+\left\|\sum_{k=1}^{m}E_{k}(T-t)\psi_{1,k}\varphi_{k}\right\|_{W_0^{s,2}(\bOm)}^2+\left\|\sum_{k=1}^{m}F_{k}(T-t)\psi_{2,k}\varphi_{k}\right\|_{L^2(\RR^{N}\setminus\Omega)}^2\right)\nonumber\\
&\leq C \left(\|\psi_0\|_{W^{s,2}(\bOm)}^2+\|\psi_1\|_{W^{s,2}(\bOm)}^2+\|\psi_2\|_{L^2(\Omega)}^2\right).\label{27.2}
\end{align} 

It follows from \eqref{27.2} that

\begin{align}\label{27.4}
\|\mathcal{N}_{s}\psi(\cdot,t)\|_{L^2(\RR^{N}\setminus\Omega)}^2\leq C\Big(\|\psi_0\|_{W^{s,2}(\bOm)}^2+\|\psi_1\|_{W^{s,2}(\bOm)}^2+\|\psi_2\|_{L^2(\Omega)}^2\Big).
\end{align}

Next, since the functions $D_n(z)$, $E_n(z)$ and $F_n(z)$ are entire functions, it follows that the function

\begin{align*}
\sum_{n=1}^{m}\Big[D_{n}(T-z)\psi_{0,n}+E_{n}(T-z)\psi_{1,n}+F_{n}(T-z)\psi_{2,n}\Big]\mathcal{N}_{s}\varphi_{n}
\end{align*}
is analytic in $\Sigma_T$. 

Let $\tau>0$ be fixed but otherwise arbitrary. Let $z\in\CC$ satisfy $\mbox{Re}(z)\le T-\tau$. Then proceeding as above by using \eqref{ES-D}, \eqref{ES-E} and \eqref{ES-F}, we get that

 \begin{align*}
& \left\Vert\sum_{n=m+1}^\infty \psi_{0,n}D_n(T-z)\mathcal N_s \varphi_n\right\Vert_{L^2(\Omc)}^2
+\left\Vert\sum_{n=m+1}^\infty \psi_{1,n}E_n(T-z)\mathcal N_s \varphi_n\right\Vert_{L^2(\Omc)}^2\\
&+\left\Vert\sum_{n=m+1}^\infty \psi_{2,n}F_n(T-z)\mathcal N_s \varphi_n\right\Vert_{L^2(\Omc)}^2\\
 \le &C\sum_{n=m+1}^\infty \left|\lambda_n^{\frac 12}\psi_{0,n}\right|^2 +\sum_{n=m+1}^\infty \left|\lambda_n^{\frac 12}\psi_{1,n}\right|^2 +C\sum_{n=m+1}^\infty |\psi_{2,n}|^2\longrightarrow 0\;\mbox{ as }\; m\to\infty.
 \end{align*}
 We have shown that
 
 \begin{align}\label{form-nor}
\mathcal N_s \psi(\cdot,z)=&\sum_{n=1}^\infty \psi_{0,n}D_n(T-z)\mathcal N_s \varphi_n+\sum_{n=1}^\infty \psi_{1,n}E_n(T-z)\mathcal N_s \varphi_n+\sum_{n=1}^\infty \psi_{2,n}F_n(T-z)\mathcal N_s \varphi_n,
 \end{align}
and the series is  convergent in $L^2(\Omc)$ uniformly in any compact subset of $\Sigma_T$. Thus, $\mathcal N_s \psi$ given in \eqref{form-nor} is also analytic in $\Sigma_T$.  The proof is finished.
\end{proof}

\section{Proof of the main controllability results }\label{prof-ma-re}

In this section we prove the main results stated in Section \ref{sec-2}. 

\subsection{The lack of exact or null controllability result}
We start with the proof of the lack of null/exact controllability of the system \eqref{SD-WE}. For this purpose, we will use the following concept of controllability.

\begin{definition}\label{esp-con}
The system \eqref{SD-WE} is said to be \emph{spectrally controllable} if any finite linear combination of eigenvectors
\begin{align*}
u_0=\sum_{n=1}^{M}u_{0,n}\varphi_{n}, \quad u_1=\sum_{n=1}^{M}u_{1,n}\varphi_{n},  \quad u_2=\sum_{n=1}^{M}u_{2,n}\varphi_{n},
\end{align*}
can be steered to zero by a control function $g$.
\end{definition}

Next, let $(u,u_t,u_{tt})$ and $(\psi,\psi_t,\psi_{tt})$ be the weak solutions of \eqref{SD-WE} and \eqref{ACP-Dual}, respectively. Multiplying the first equation in \eqref{SD-WE} by $\psi$, then integrating by parts over $(0,T)$ and over $\Om$ and using the integration by parts formulas \eqref{Int-Part}-\eqref{Int-Part-2}, we get 

\begin{align}\label{32}
\int_\Omega& \Big(u_{tt}\psi-u_t\psi_{t}+ u\psi_{tt}+\alpha(u_t\psi-u\psi_t)+b u(-\Delta)^s \psi\Big)\Big|_{t=0}^{t=T}dx\notag\\
&=\int_0^{T}\int_{\Omc}\Big(c^2 g(x,t)+b g_t(x,t)\Big)\mathcal{N}_{s}\psi(x,t)dxdt.
\end{align} 

Using the identity \eqref{32} and a density argument to pass to the limit, we obtain the following criterion of null and exact controllabilities.

\begin{lemma}\label{L1}
The following assertions hold.
\begin{enumerate}
\item  The system \eqref{SD-WE} is null controllable if and only if for each initial condition $(u_0,u_1,u_2)\in W_0^{s,2}(\bOm)\times W_0^{s,2}(\bOm)\times L^2(\Omega) $, there exists a control function $g\in L^2((0,T);W^{s,2}(\Omc))$ such that the weak solution $(\psi,\psi_t,\psi_{tt})$ of the dual system \eqref{ACP-Dual} satisfies
\begin{align}\label{33}
&-(u_2,\psi(\cdot,0))_{L^2(\Om)}+\langle u_1,\psi_{t}(\cdot,0)\rangle_{\frac 12,-\frac 12}-\langle u_0,\psi_{tt}(\cdot,0)\rangle_{\frac 12,-\frac 12}\notag\\
&-\alpha (u_1,\psi(\cdot,0))_{L^2(\Om)}+\alpha \langle u_0,\psi_{t}(\cdot,0)\rangle_{\frac 12,-\frac 12}-b\langle u_0,(-\Delta)^s\psi(\cdot,0)\rangle_{\frac 12,-\frac 12}\notag\\
=&\int_0^{T}\int_{\Omc}\Big(c^2 g(x,t)+b g_t(x,t)\Big)\mathcal{N}_{s}\psi(x,t)dxdt,
\end{align}
for each $(\psi_0,\psi_1,\psi_2)\in L^2(\Omega)\times W^{-s,2}(\bOm)\times W^{-s,2}(\bOm)$.

\item The system \eqref{SD-WE} is exact controllable at time $T>0$, if and only if there exists a control function $g\in L^2((0,T);W^{s,2}(\Omc))$ such that the solution $(\psi,\psi_t,\psi_{tt})$ of \eqref{ACP-Dual} satisfies
\begin{align}\label{33-2}
&(u_{tt}(\cdot,T),\psi_0)_{L^2(\Om)}-\langle u_t(\cdot,T),\psi_1\rangle_{\frac 12,-\frac 12}+\langle u(\cdot,T),\psi_2\rangle_{\frac 12,-\frac 12}\notag\\
&+\alpha\Big( (u_t(\cdot,T),\psi_0)_{L^2(\Om)}- \langle u(\cdot,T),\psi_1\rangle_{\frac 12,-\frac 12}\Big)+b\langle u(\cdot,T),(-\Delta)^s\psi_0\rangle_{\frac 12,-\frac 12}\notag\\
=&\int_0^{T}\int_{\Omc}\Big(c^2 g(x,t)+b g_t(x,t)\Big)\mathcal{N}_{s}\psi(x,t)dxdt,
\end{align}
for each $(\psi_0,\psi_1,\psi_2)\in L^2(\Omega)\times W^{-s,2}(\bOm)\times W^{-s,2}(\bOm)$.
\end{enumerate}
\end{lemma}

Now, we are able to give the proof of the first main result of this article.

\begin{proof}[{\bf Proof of Theorem \ref{lact-nul-cont}}]
Using Definition \ref{esp-con}, we prove that no non-trivial finite linear combination of eigenvectors can be driven to zero in finite time. 

Write the initial data in Fourier series
\begin{align}\label{37}
u_0=\sum_{n=1}^{\infty}u_{0,n}\varphi_{n}, \quad u_1=\sum_{n=1}^{\infty}u_{1,n}\varphi_{n}, \quad u_2=\sum_{n=1}^{\infty}u_{2,n}\varphi_{n},
\end{align}
and suppose that there exists $M\in\NN$ such that  
\begin{align}\label{37.1}
u_{0,n}=u_{1,n}=u_{2,n}=0, \quad \forall\;  n\geq M.
\end{align}

Assume that the system \eqref{SD-WE} is spectrally controllable. Then, there exists a control function $g$ such that the solution $(u,u_t,u_{tt})$ of \eqref{SD-WE} with $u_0,u_1,u_2$ given by \eqref{37}--\eqref{37.1} satisfy $u(\cdot,T)=u_{t}(\cdot,T)=u_{tt}(\cdot,T)=0$ in $\Omega$. From Lemma \ref{L1} we have
\begin{align}\label{38}
&-(u_2,\psi(\cdot,0))_{L^2(\Om)}+\langle u_1,\psi_{t}(\cdot,0)\rangle_{\frac 12,-\frac 12}-\langle u_0,\psi_{tt}(\cdot,0)\rangle_{\frac 12,-\frac 12}\notag\\
&-\alpha (u_1,\psi(\cdot,0))_{L^2(\Om)}+\alpha \langle u_0,\psi_{t}(\cdot,0)\rangle_{\frac 12,-\frac 12}-b\langle u_0,(-\Delta)^s\psi(\cdot,0)\rangle_{\frac 12,-\frac 12}\notag\\
=&\int_0^{T}\int_{\Omc}\Big(c^2 g(x,t)+b g_t(x,t)\Big)\mathcal{N}_{s}\psi(x,t)dxdt,
\end{align}
for any solution $(\psi,\psi_t,\psi_{tt})$ of the dual system \eqref{ACP-Dual}.

We consider the following trajectories: 
\begin{align}\label{38.1}
\psi(x,t)=e^{\lambda_{n,j}(T-t)}\varphi_{n}(x), \quad j=1,2,3 .
\end{align}
Replacing \eqref{38.1} in \eqref{38} we obtain, for any $n\in[1,M-1]$, the following system:
\begin{multline}\label{39}
-u_{2,n}e^{\lambda_{n,j}T}+u_{1,n}\lambda_{n,j}e^{\lambda_{n,j}T}-u_{0,n}\lambda_{n,j}^2 e^{\lambda_{n,j}T}-\alpha e^{\lambda_{n,j}T}\Big(u_{1,n}-u_{0,n}\lambda_{n,i}\Big)-b  u_{0,n} \lambda_{n}e^{\lambda_{n,j }T}\\=\int_0^{T}\int_{\Omc}(c^2 g(x,t)+b g_t(x,t))e^{\lambda_{n,j}(T- t)}\mathcal{N}_{s}\varphi_{n}(x)dxdt.
\end{multline}
Multiplying \eqref{39} by $e^{-\lambda_{n,j}T}$, for each $j=1,2,3$, we obtain that the moment problem is to find some $g$ that satisfies
\begin{multline}\label{40}
-u_{2,n}+u_{1,n}\lambda_{n,j}-u_{0,n}\lambda_{n,j}^2 -\alpha \Big(u_{1,n}-u_{0,n}\lambda_{n,j}\Big)-b  u_{0,n} \lambda_{n}\\=\int_0^{T}\int_{\Omc}(c^2 g(x,t)+b g_t(x,t))e^{-\lambda_{n,j}t}\mathcal{N}_{s}\varphi_{n}(x)dxdt.
\end{multline}

Next, following the works \cite{martin2013null,WZ}, we define the complex function
\begin{align}\label{43}
F(z)=\int_0^{T}\left(\int_{\Omc}(c^2 g(x,t)+b g_t(x,t))\mathcal{N}_{s}\varphi_{n}(x)dx\right)e^{izt}dt.
\end{align}
According to the Paley--Wiener theorem, $F$ is an entire function. Due to \eqref{37.1}, from \eqref{40} we obtain that $F$ satisfies $F(i\lambda_{n,j})=0$, for all $n\ge M$. Besides, we know that $\lambda_{n,1}\to -\frac{c^2}{b}$ as $n\to\infty$ (see Lemma \ref{eigen-values-functions}). Then, $F$ is zero in a set with finite accumulation point. This implies that $F\equiv 0$.  It follows from \eqref{40} and \eqref{43} that

\begin{align*}
\underbrace{\left(\begin{array}{ccc}
\D\alpha\lambda_{n,1}-\lambda_{n,1}^2-b\lambda_n&\D \lambda_{n,1}-\alpha &\D -1\vspace*{0.2cm}\\\vspace*{0.2cm}
\D\alpha\lambda_{n,2}-\lambda_{n,2}^2-b\lambda_n& \D\lambda_{n,2}-\alpha & -1\\
\D\alpha\lambda_{n,3}-\lambda_{n,3}^2-b\lambda_n& \D\lambda_{n,3}-\alpha &\D-1
\end{array}\right)}_{B} \left(\begin{array}{c}
u_{0,n}\vspace*{0.2cm}\\\vspace*{0.2cm}
u_{1,n}\\\vspace*{0.2cm}
u_{2,n}
\end{array}\right)=\left(\begin{array}{c}
0\vspace*{0.2cm}\\\vspace*{0.2cm}
0\\
0
\end{array}\right).
\end{align*}
Calculating we get that
\begin{align*}
\mbox{det}(B)=(\lambda_{n,1}-\lambda_{n,2})(\lambda_{n,1}-\lambda_{n,3})(\lambda_{n,2}-\lambda_{n,3})\ne 0.
\end{align*}
Hence, the matrix $B$ is invertible and we can then conclude that
$u_{0,n}=u_{1,n}=u_{2,n}=0$, for $n<M$. Thus the trivial state is the only one which can be steered to zero. We have shown that the system is not spectrally controllable. It is clear from the proof that this implies that the system is not exact or null controllable.
The proof is finished.
\end{proof}

\subsection{The unique continuation property}

\begin{proof}[{\bf Proof of Theorem \ref{pro-uni-con}}]
Let $\mathcal O\subset\Omc$ be an arbitrary non-empty open set. Suppose that $\mathcal{N}_{s}\psi=0$ in $\mathcal{O}\times(0,T)$. Then, for all $(x,t)\in\mathcal{O}\times(0,T)$ we have that
\begin{align}\label{44}
\mathcal{N}_{s}\psi(x,t)=\sum_{n=1}^{\infty}\Big(D_{n}(T-t)\psi_{0,n}+E_{n}(T-t)\psi_{1,n}+F_{n}(T-t)\psi_{2,n}\Big)\mathcal{N}_{s}\varphi_{n}(x)=0.
\end{align}

Since $\mathcal{N}_{s}\psi$ can be analytically extended to $\Sigma_{T}$, it follows that  for all $(x,t)\in\mathcal{O}\times(-\infty,T)$,
\begin{align}\label{45}
\mathcal{N}_{s}\psi(x,t)=\sum_{n=1}^{\infty}\Big(D_{n}(T-t)\psi_{0,n}+E_{n}(T-t)\psi_{1,n}+F_{n}(T-t)\psi_{2,n}\Big)\mathcal{N}_{s}\varphi_{n}(x)=0.
\end{align}

Let $\{\lambda_{k}\}_{k\in\NN}$ be the set of all eigenvalues of the operator $(-\Delta)_D^{s}$ and let $\{\varphi_{k_{j}}\}_{1\leq j\leq k_{j}}$ be an orthonormal basis for ker$(\lambda_{k}-(-\Delta)_D^{s})$. Then, \eqref{45} can be rewritten as 

\begin{align}\label{46}
\mathcal{N}_{s}\psi(x,t)=&\sum_{k=1}^{\infty}\left(\sum_{j=1}^{m_{k}}\psi_{0,k_{j}}\mathcal{N}_{s}\varphi_{k_{j}}(x)\right)D_{k}(T-t)+\sum_{k=1}^{\infty}\left(\sum_{j=1}^{m_{k}}\psi_{1,k_{j}}\mathcal{N}_{s}\varphi_{k_{j}}(x)\right)E_{k}(T-t)\notag\\
&+\sum_{k=1}^{\infty}\left(\sum_{j=1}^{m_{k}}\psi_{2,k_{j}}\mathcal{N}_{s}\varphi_{k_{j}}(x)\right)F_{k}(T-t)=0,\quad \forall\; (x,t)\in\mathcal{O}\times(-\infty,T).
\end{align}

Let $z\in\CC$ with Re$(z)=\eta>0$ and let $m\in\NN$. Since $\varphi_{k_{j}}$, $1\leq k\leq m$, are orthonormal, then using the fact that the operator $\mathcal{N}_{s}:D((-\Delta)_D^{s})\subset W^{s,2}(\RR^{N})\to L^2(\RR^{N}\setminus\Omega)$ is bounded and the continuous dependence on the data of $\mathcal{N}_{s}$ (see \eqref{27.4}), and letting 

\begin{align*}
\psi_{m}(\cdot,t):=&\sum_{k=1}^{m}\left(\sum_{j=1}^{m_{k}}\psi_{0,k_{j}}\mathcal{N}_{s}\varphi_{k_{j}}(x)\right)e^{z(t-T)}D_{k}(T-t)+\sum_{k=1}^{m}\left(\sum_{j=1}^{m_{k}}\psi_{1,k_{j}}\mathcal{N}_{s}\varphi_{k_{j}}(x)\right)e^{z(t-T)}E_{k}(T-t) \notag\\
&+\sum_{k=1}^{m}\left(\sum_{j=1}^{m_{k}}\psi_{2,k_{j}}\mathcal{N}_{s}\varphi_{k_{j}}(x)\right)e^{z(t-T)}F_{k}(T-t),
\end{align*}
we obtain that there is a constant $C>0$ such that for every $t\in [0,T]$,

\begin{align}\label{48}
\|\psi_{m}(\cdot,t)\|_{L^2(\RR^{N}\setminus\Omega)}\leq C e^{\eta(t-T)}\left(\|\psi_0\|_{W^{s,2}(\bOm)}+\|\psi_1\|_{W^{s,2}(\bOm)}+\|\psi_2\|_{L^2(\Omega)}\right).
\end{align}
The right hand side of \eqref{48} is integrable over $t\in(-\infty,T)$ and
\begin{align*}
\int_{-\infty}^{T}e^{\eta(t-T)}(\|\psi_0\|_{W^{s,2}(\bOm)}+\|\psi_1\|_{W^{s,2}(\bOm)}+\|\psi_2\|_{L^2(\Omega)})dt=\frac{1}{\eta}\left(\|\psi_0\|_{W^{s,2}(\bOm)}+\|\psi_1\|_{W^{s,2}(\bOm)}+\|\psi_2\|_{L^2(\Omega)}\right).
\end{align*}

By the Lebesgue dominated convergence theorem, we can deduce that
\begin{align*}
\int_{-\infty}^{T}e^{z(t-T)}&\left[\sum_{k=1}^{\infty}\left(\sum_{j=1}^{m_{k}}\psi_{0,k_{j}}\mathcal{N}_{s}\varphi_{k_{j}}(x)\right)D_{k}(T-t)+\sum_{k=1}^{\infty}\left(\sum_{j=1}^{m_{k}}\psi_{1,k_{j}}\mathcal{N}_{s}\varphi_{k_{j}}(x)\right)E_{k}(T-t)\right.\\
&+\left.\sum_{k=1}^{\infty}\left(\sum_{j=1}^{m_{k}}\psi_{2,k_{j}}\mathcal{N}_{s}\varphi_{k_{j}}(x)\right)F_{k}(T-t)\right]dt\\
=&\sum_{k=1}^{\infty}\sum_{j=1}^{m_{k}}\Big(G_{k}(z)\psi_{0,k_{j}}+H_{k}(z)\psi_{1,k_{j}}+I_{k}(z)\psi_{2,k_{j}}\Big)\mathcal{N}_{s}\varphi_{k_{j}}(x), \quad x\in\RR^{N}\setminus\Omega, \ \text{Re}(z)>0,
\end{align*}
where
\begin{align}\label{Gn}
\left\{\begin{array}{c}
G_{k}(z)=\D\frac{\lambda_{k,2}\lambda_{k,3}}{\xi_{k,1}(z-\lambda_{k,1})}-\frac{\lambda_{k,1}\lambda_{k,3}}{\xi_{k,2}(z-\lambda_{k,2})}+\frac{\lambda_{k,1}\lambda_{k,2}}{\xi_{k,3}(z-\lambda_{k,3})},\\
H_{k}(z)=\D-\frac{\lambda_{k,2}+\lambda_{k,3}}{\xi_{k,1}(z-\lambda_{k,1})}+\frac{\lambda_{k,1}+\lambda_{k,3}}{\xi_{k,2}(z-\lambda_{k,2})}-\frac{\lambda_{k,1}+\lambda_{k,2}}{\xi_{k,3}(z-\lambda_{k,3})},\\
I_{k}(z)=\D\frac{1}{\xi_{k,1}(z-\lambda_{k,1})}-\frac{1}{\xi_{k,2}(z-\lambda_{k,2})}+\frac{1}{\xi_{k,3}(z-\lambda_{k,3})}.
\end{array}\right.
\end{align}
We recall that $\xi_{k,1}$, $\xi_{k,2}$ and $\xi_{k,3}$ are given in \eqref{xi}.

From \eqref{46} we get that
\begin{align}\label{50}
\sum_{k=1}^{\infty}\sum_{j=1}^{m_{k}}\Big(G_{k}(z)\psi_{0,k_{j}}+H_{k}(z)\psi_{1,k_{j}}+I_{k}(z)\psi_{2,k_{j}}\Big)\mathcal{N}_{s}\varphi_{k_{j}}(x)=0,\quad x\in\mathcal{O}, \ \text{Re}(z)>0.
\end{align}
Using the analytic continuation in $z$, we obtain that \eqref{50} holds for every $z\in\CC\setminus\{\lambda_{k,1},\lambda_{k,2},\lambda_{k,3}\}_{k\in\NN}$.  

We take a small circle about $\lambda_{k,h}$, for some $h\in\{1,2,3\}$,  but not including $\{\lambda_{l,j}\}_{l\neq k,\; j\neq h}$, with $j\in\{1,2,3\}$. Then, integrating over that circle we get the following system:
\begin{align}
\sum_{j=1}^{m_{k}}\left[\frac{\lambda_{k_j,2}\lambda_{k_j,3}}{\xi_{k_j,1}}\psi_{0,k_{j}}-\frac{\lambda_{k_j,2}+\lambda_{k_j,3}}{\xi_{k_j,1}}\psi_{1,k_j}+\frac{1}{\xi_{k_j,1}}\psi_{2,k_j}\right]\mathcal{N}_{s}\varphi_{k_{j}}(x)=&0,\quad x\in\mathcal{O},\label{51}\\
\sum_{j=1}^{m_{k}}\left[\frac{-\lambda_{k_j,1}\lambda_{k_j,3}}{\xi_{k_j,2}}\psi_{0,k_{j}}+\frac{\lambda_{k_j,1}+\lambda_{k_j,3}}{\xi_{k_j,2}}\psi_{1,k_j}-\frac{1}{\xi_{k_j,2}}\psi_{2,k_j}\right]\mathcal{N}_{s}\varphi_{k_{j}}(x)=&0,\quad x\in\mathcal{O},\label{51-a}\\
\sum_{j=1}^{m_{k}}\left[\frac{\lambda_{k_j,1}\lambda_{k_j,2}}{\xi_{k_j,3}}\psi_{0,k_{j}}-\frac{\lambda_{k_j,1}+\lambda_{k_j,2}}{\xi_{k_j,3}}\psi_{1,k_j}+\frac{1}{\xi_{k_j,3}}\psi_{2,k_j}\right]\mathcal{N}_{s}\varphi_{k_{j}}(x)=&0,\quad x\in\mathcal{O}.\label{51-b}
\end{align}

Let
\begin{align*}
\psi_{k}^1&:=\sum_{j=1}^{m_{k}}\left[\frac{\lambda_{k_j,2}\lambda_{k_j,3}}{\xi_{k_j,1}}\psi_{0,k_{j}}-\frac{\lambda_{k_j,2}+\lambda_{k_j,3}}{\xi_{k_j,1}}\psi_{1,k_j}+\frac{1}{\xi_{k_j,1}}\psi_{2,k_j}\right]\varphi_{k_{j}},\\
\psi_{k}^2&=\sum_{j=1}^{m_{k}}\left[\frac{-\lambda_{k_j,1}\lambda_{k_j,3}}{\xi_{k_j,2}}\psi_{0,k_{j}}+\frac{\lambda_{k_j,1}+\lambda_{k_j,3}}{\xi_{k_j,1}}\psi_{1,k_j}-\frac{1}{\xi_{k_j,2}}\psi_{2,k_j}\right]\varphi_{k_{j}},\\
\psi_k^3&=\sum_{j=1}^{m_{k}}\left[\frac{\lambda_{k_j,1}\lambda_{k_j,2}}{\xi_{k_j,3}}\psi_{0,k_{j}}-\frac{\lambda_{k_j,1}+\lambda_{k_j,2}}{\xi_{k_j,3}}\psi_{1,k_j}+\frac{1}{\xi_{k_j,3}}\psi_{2,k_j}\right]\varphi_{k_{j}}.
\end{align*}
It follows from \eqref{51}, \eqref{51-a} and \eqref{51-b} that $\mathcal{N}_{s}\psi_k^1=\mathcal{N}_{s}\psi_k^2=\mathcal{N}_{s}\psi_k^3=0$ in $\mathcal{O}$. We have shown that 
\begin{align*}
(-\Delta)^{s}\psi_{k}^l=\lambda_{k}\psi_{k}^l \quad\text{  in }\Omega\quad\text{ and }\quad\mathcal{N}_{s}\psi_{k}^l=0\quad\text{ in }\mathcal{O}, \ l=1,2,3.
\end{align*}
From Lemma \ref{lem-UCD}, we deduce that $\psi_k^l=0$, for every $k\in\NN$ and $l=1,2,3$. Since the system $\{\varphi_{k_j}\}_{1\leq j\leq m_k}$ is linearly independent in $L^2(\Omega)$, we get that 
\begin{align*}
\underbrace{\left(\begin{array}{ccc}
\D \frac{\lambda_{k,2}\lambda_{k,3}}{\xi_{k,1}}& -\D\frac{\lambda_{k,2}+\lambda_{k,3}}{\xi_{k,1}} &\D \frac{1}{\xi_{k,1}}\vspace*{0.2cm}\\\vspace*{0.2cm}
\D\frac{-\lambda_{k,1}\lambda_{k,3}}{\xi_{k,2}} & \D\frac{\lambda_{k,1}+\lambda_{k,3}}{\xi_{k,2}} & -\D\frac{1}{\xi_{k,2}}\\
\D\frac{\lambda_{k,1}\lambda_{k,2}}{\xi_{k,3}}& -\D\frac{\lambda_{k,1}+\lambda_{k,2}}{\xi_{k,3}} &\D\frac{1}{\xi_{k,3}}
\end{array}\right)}_{A} \left(\begin{array}{c}
\psi_{0,k}\vspace*{0.8cm}\\\vspace*{0.8cm}
\psi_{1,k}\\
\psi_{2,k}
\end{array}\right)=\left(\begin{array}{c}
0\vspace*{0.8cm}\\\vspace*{0.8cm}
0\\
0
\end{array}\right).
\end{align*}
A simple calculation shows that the determinant of the matrix $A$ is given by
\begin{align*}
\det(A)=\frac{i}{2\operatorname{Im}(\lambda_{k,2})\left[\operatorname{Re}(\lambda_{k,2})-\lambda_{k,1})^2+(\operatorname{Im}(\lambda_{k,2}))^2\right]}\ne 0.
\end{align*}
Since the matrix $A$ is invertible, we can deduce that 
\begin{align*}
\psi_{0,k}=\psi_{1,k}=\psi_{2,k}=0, \  k\in\NN.
\end{align*}
Since the solution $(\psi,\psi_t,\psi_{tt})$ of the adjoint system is unique, we can conclude that $\psi=0$ in $\Omega\times (0,T)$. The proof is finished.
\end{proof}

\subsection{The approximate controllability}

We obtain the result as a direct consequence of the unique continuation property for the adjoint system (Theorem \ref{pro-uni-con}).

\begin{proof}[{\bf Proof of Theorem \ref{main-Theo}}]
Let $g\in \mathcal D(\mathcal O\times (0,T))$,  $(u,u_t,u_{tt})$ the unique weak solution of \eqref{SD-WE} with $u_0=u_1=u_2=0$ and let $(\psi,\psi_t,\psi_{tt})$ be the unique weak solution of \eqref{ACP-Dual} with $(\psi_0,\psi_1,\psi_2)\in W_0^{s,2}(\bOm)\times W_0^{s,2}(\bOm)\times L^2(\Omega)$. Firstly, it follows from Theorems \ref{theo-44}  that $u\in C^\infty([0,T];W^{s,2}(\RR^N))$. Thus $u(\cdot,T)\in L^2(\Omega)$, $u_t(\cdot,T)\in W^{-s,2}(\bOm)$ and $u_{tt}(\cdot,T)\in W^{-s,2}(\bOm)$.
Secondly, it follows from Theorem \ref{theo-48} that $\psi\in L^1((0,T);L^2(\Omega))$.   
Therefore, using the identity  \eqref{32} we can deduce  that

\begin{align}\label{58}
&\langle u_{tt}(\cdot,T),\psi_0\rangle_{\frac 12,-\frac 12}-\langle u_t(\cdot,T),\psi_1\rangle_{\frac 12,-\frac 12}+ (u(\cdot,T),\psi_2)_{L^2(\Omega)}\notag\\
&+\alpha\Big( (u_t(\cdot,T),\psi_0)_{\frac 12,-\frac 12}- \langle u(\cdot,T),\psi_1\rangle_{\frac 12,-\frac 12}\Big)+b\langle u(\cdot,T),(-\Delta)^s\psi_0\rangle_{\frac 12,-\frac 12}\notag\\
=&\int_0^{T}\int_{\Omc}\Big(c^2 g(x,t)+b g_t(x,t)\Big)\mathcal{N}_{s}\psi(x,t)dxdt,
\end{align}

If $(\psi_0,\psi_1,\psi_2)\in D((-\Delta)_D^s)\times ((-\Delta)_D^s)\times L^2(\Om)\hookrightarrow W_0^{s,2}(\bOm)\times W_0^{s,2}(\bOm)\times  L^2(\Omega)$, then \eqref{58} becomes
\begin{align}\label{58-2}
\langle u_{tt}(\cdot,T),\psi_0\rangle_{\frac 12,-\frac 12}&+\langle u_t(\cdot,T),\alpha\psi_0-\psi_1\rangle_{\frac 12,-\frac 12}+ \Big(u(\cdot,T),-\alpha\psi_1+\psi_2+b(-\Delta)^s\psi_0\Big)_{L^2(\Omega)}\notag\\
&=\int_0^{T}\int_{\Omc}\Big(c^2 g(x,t)+b g_t(x,t)\Big)\mathcal{N}_{s}\psi(x,t)dxdt.
\end{align}
Since $D((-\Delta)_D^s)\times D((-\Delta)_D^s)\times L^2(\Om)$ is dense in $W_0^{s,2}(\bOm)\times  W_0^{s,2}(\bOm)\times  L^2(\Omega)$, to prove that the set $\Big\{(u(\cdot,T),u_t(\cdot,T),u_{tt}(\cdot,T)):\; g\in \mathcal D(\mathcal O\times (0,T))\Big\}$ is dense in $L^2(\Omega)\times W^{-s,2}(\bOm)\times W^{-s,2}(\bOm)$, it suffices to show that if $(\psi_0,\psi_1,\psi_2)\in D((-\Delta)_D^s)\times D((-\Delta)_D^s)\times  L^2(\Om)$ is such that
\begin{align}\label{eq42}
\langle u_{tt}(\cdot,T),\psi_0\rangle_{\frac 12,-\frac 12}+\langle u_t(\cdot,T),\alpha\psi_0-\psi_1\rangle_{\frac 12,-\frac 12}+ \Big(u(\cdot,T),-\alpha\psi_1+\psi_2+b(-\Delta)^s\psi_0\Big)_{L^2(\Omega)}=0,
\end{align}
for any $g\in \mathcal D(\mathcal O\times (0,T))$, then $\psi_0=\psi_1=\psi_2=0$. 

Indeed, let $(\psi_0,\psi_1,\psi_2)\in D((-\Delta)_D^s)\times D((-\Delta)_D^s)\times L^2(\Om)$ satisfy \eqref{eq42}. It follows from \eqref{58-2} and \eqref{eq42} that
\begin{align*}
\int_0^{T}\int_{\Omc}\Big(c^2 g(x,t)+b g_t(x,t)\Big)\mathcal{N}_{s}\psi(x,t)dxdt=0,
\end{align*}
for any $g\in \mathcal D(\mathcal O\times (0,T))$. Recall that $b>0$. By the fundamental lemma of the calculus of variations, we can deduce that
\begin{align*}
\mathcal N_s\psi=0\;\;\mbox{ in }\; \mathcal O\times (0,T).
\end{align*}
It follows from Theorem \ref{pro-uni-con} that $\psi=0$ in $\Omega\times(0,T)$. Since the solution $(\psi,\psi_t,\psi_{tt})$ of \eqref{ACP-Dual} is unique, we can conclude that $\psi_0=\psi_1=\psi_2=0$ in $\Omega$. The proof is finished.
\end{proof}



\bibliographystyle{plain}
\bibliography{biblio}

\begin{thebibliography}{10}

\bibitem{ATW}
W.~Arendt, A.~F.~M. ter Elst, and M.~Warma.
\newblock Fractional powers of sectorial operators via the
  {D}irichlet-to-{N}eumann operator.
\newblock {\em Comm. Partial Differential Equations}, 43(1):1--24, 2018.

\bibitem{BWZ2}
U.~Biccari, M.~Warma, and E.~Zuazua.
\newblock Addendum: {L}ocal elliptic regularity for the {D}irichlet fractional
  {L}aplacian.
\newblock {\em Adv. Nonlinear Stud.}, 17(4):837--839, 2017.

\bibitem{BWZ1}
U.~Biccari, M.~Warma, and E.~Zuazua.
\newblock Local elliptic regularity for the {D}irichlet fractional {L}aplacian.
\newblock {\em Adv. Nonlinear Stud.}, 17(2):387--409, 2017.

\bibitem{BBC}
K.~Bogdan, K.~Burdzy, and Z-Q. Chen.
\newblock Censored stable processes.
\newblock {\em Probab. Theory Related Fields}, 127(1):89--152, 2003.

\bibitem{Caf1}
L.~A. Caffarelli, J-M. Roquejoffre, and Y.~Sire.
\newblock Variational problems for free boundaries for the fractional
  {L}aplacian.
\newblock {\em J. Eur. Math. Soc.}, 12(5):1151--1179, 2010.

\bibitem{CLD}
A.~H. Caixeta, I.~Lasiecka, and V.~N. Domingos~Cavalcanti.
\newblock On long time behavior of {M}oore-{G}ibson-{T}hompson equation with
  molecular relaxation.
\newblock {\em Evol. Equ. Control Theory}, 5(4):661--676, 2016.

\bibitem{DLP}
F.~Dell'Oro, I.~Lasiecka, and V.~Pata.
\newblock The {M}oore-{G}ibson-{T}hompson equation with memory in the critical
  case.
\newblock {\em J. Differential Equations}, 261(7):4188--4222, 2016.

\bibitem{NPV}
E.~Di~Nezza, G.~Palatucci, and E.~Valdinoci.
\newblock Hitchhiker's guide to the fractional {S}obolev spaces.
\newblock {\em Bull. Sci. Math.}, 136(5):521--573, 2012.

\bibitem{DRV}
S.~Dipierro, X.~Ros-Oton, and E.~Valdinoci.
\newblock Nonlocal problems with {N}eumann boundary conditions.
\newblock {\em Rev. Mat. Iberoam.}, 33(2):377--416, 2017.

\bibitem{DS}
A.~A. Dubkov, B.~Spagnolo, and V.~V. Uchaikin.
\newblock L\'evy flight superdiffusion: an introduction.
\newblock {\em Internat. J. Bifur. Chaos Appl. Sci. Engrg.}, 18(9):2649--2672,
  2008.

\bibitem{GW}
C.~G. Gal and M.~Warma.
\newblock Bounded solutions for nonlocal boundary value problems on {L}ipschitz
  manifolds with boundary.
\newblock {\em Adv. Nonlinear Stud.}, 16(3):529--550, 2016.

\bibitem{GW-F}
C.~G. Gal and M.~Warma.
\newblock Fractional in time semilinear parabolic equations and applications.
\newblock {\em HAL Id: hal-01578788}, 2017.

\bibitem{GW-CPDE}
C.~G. Gal and M.~Warma.
\newblock Nonlocal transmission problems with fractional diffusion and boundary
  conditions on non-smooth interfaces.
\newblock {\em Comm. Partial Differential Equations}, 42(4):579--625, 2017.

\bibitem{GW2}
C.~G. Gal and M.~Warma.
\newblock On some degenerate non-local parabolic equation associated with the
  fractional {$p$}-{L}aplacian.
\newblock {\em Dyn. Partial Differ. Equ.}, 14(1):47--77, 2017.

\bibitem{GSU}
T.~Ghosh, M.~Salo, and G.~Uhlmann.
\newblock The {C}alder\'on problem for the fractional {S}chr\" odinger
  equation.
\newblock {\em arXiv:1609.09248}.

\bibitem{GR}
R.~Gorenflo, F.~Mainardi, and A.~Vivoli.
\newblock Continuous-time random walk and parametric subordination in
  fractional diffusion.
\newblock {\em Chaos Solitons Fractals}, 34(1):87--103, 2007.

\bibitem{Gris}
P.~Grisvard.
\newblock {\em Elliptic problems in nonsmooth domains}, volume~69 of {\em
  Classics in Applied Mathematics}.
\newblock Society for Industrial and Applied Mathematics (SIAM), Philadelphia,
  PA, 2011.
\newblock Reprint of the 1985 original. With a foreword by Susanne C. Brenner.

\bibitem{Grub}
G.~Grubb.
\newblock Fractional {L}aplacians on domains, a development of {H}\"ormander's
  theory of {$\mu$}-transmission pseudodifferential operators.
\newblock {\em Adv. Math.}, 268:478--528, 2015.

\bibitem{JW}
A.~Jonsson and H.~Wallin.
\newblock Function spaces on subsets of {${\bf R}^n$}.
\newblock {\em Math. Rep.}, 2(1):xiv+221, 1984.

\bibitem{KL}
B.~Kaltenbacher and I.~Lasiecka.
\newblock Exponential decay for low and higher energies in the third order
  linear {M}oore-{G}ibson-{T}hompson equation with variable viscosity.
\newblock {\em Palest. J. Math.}, 1(1):1--10, 2012.

\bibitem{KLM}
B.~Kaltenbacher, I.~Lasiecka, and R.~Marchand.
\newblock Wellposedness and exponential decay rates for the
  {M}oore-{G}ibson-{T}hompson equation arising in high intensity ultrasound.
\newblock {\em Control Cybernet.}, 40(4):971--988, 2011.

\bibitem{KLP}
B.~Kaltenbacher, I.~Lasiecka, and M.~K. Pospieszalska.
\newblock Well-posedness and exponential decay of the energy in the nonlinear
  {J}ordan-{M}oore-{G}ibson-{T}hompson equation arising in high intensity
  ultrasound.
\newblock {\em Math. Models Methods Appl. Sci.}, 22(11):1250035, 34, 2012.

\bibitem{KW}
V.~Keyantuo and M.~Warma.
\newblock On the interior approximate controllability for fractional wave
  equations.
\newblock {\em Discrete Contin. Dyn. Syst.}, 36(7):3719--3739, 2016.

\bibitem{Las}
I.~Lasiecka.
\newblock Global solvability of {M}oore-{G}ibson-{T}hompson equation with
  memory arising in nonlinear acoustics.
\newblock {\em J. Evol. Equ.}, 17(1):411--441, 2017.

\bibitem{LW}
I.~Lasiecka and X.~Wang.
\newblock Moore-{G}ibson-{T}hompson equation with memory, part {II}: {G}eneral
  decay of energy.
\newblock {\em J. Differential Equations}, 259(12):7610--7635, 2015.

\bibitem{LW2}
I.~Lasiecka and X.~Wang.
\newblock Moore-{G}ibson-{T}hompson equation with memory, part {I}: exponential
  decay of energy.
\newblock {\em Z. Angew. Math. Phys.}, 67(2):Art. 17, 23, 2016.

\bibitem{LiZa2018}
C.~Lizama and S.~Zamorano.
\newblock Controllability results for the {M}oore--{G}ibson--{T}hompson
  equation arising in nonlinear acoustics.
\newblock {\em J. Differential Equations}, 2018, to appear.

\bibitem{CLR-MW}
C.~Louis-Rose and M.~Warma.
\newblock Approximate controllability from the exterior of space-time
  fractional wave equations.
\newblock {\em Applied Mathematics \& Optimization}, pages 1--44, 2018.

\bibitem{Man}
B.~B. Mandelbrot and J.~W. Van~Ness.
\newblock Fractional {B}rownian motions, fractional noises and applications.
\newblock {\em SIAM Rev.}, 10:422--437, 1968.

\bibitem{marchand2012abstract}
R~Marchand, T~McDevitt, and R~Triggiani.
\newblock An abstract semigroup approach to the third-order
  {Moore--Gibson--Thompson} partial differential equation arising in
  high-intensity ultrasound: structural decomposition, spectral analysis,
  exponential stability.
\newblock {\em Mathematical Methods in the Applied Sciences},
  35(15):1896--1929, 2012.

\bibitem{martin2013null}
P.~Martin, L.~Rosier, and P.~Rouchon.
\newblock Null controllability of the structurally damped wave equation with
  moving control.
\newblock {\em SIAM Journal on Control and Optimization}, 51(1):660--684, 2013.

\bibitem{RS2-2}
X.~Ros-Oton and J.~Serra.
\newblock The {D}irichlet problem for the fractional {L}aplacian: regularity up
  to the boundary.
\newblock {\em J. Math. Pures Appl. (9)}, 101(3):275--302, 2014.

\bibitem{RS-DP}
X.~Ros-Oton and J.~Serra.
\newblock The extremal solution for the fractional {L}aplacian.
\newblock {\em Calc. Var. Partial Differential Equations}, 50(3-4):723--750,
  2014.

\bibitem{rosier2007}
L.~Rosier and P.~Rouchon.
\newblock On the controllability of a wave equation with structural damping.
\newblock {\em Int. J. Tomogr. Stat}, 5(W07):79--84, 2007.

\bibitem{Sch}
W.~R. Schneider.
\newblock Grey noise.
\newblock In {\em Stochastic processes, physics and geometry ({A}scona and
  {L}ocarno, 1988)}, pages 676--681. World Sci. Publ., Teaneck, NJ, 1990.

\bibitem{Val}
E.~Valdinoci.
\newblock From the long jump random walk to the fractional {L}aplacian.
\newblock {\em Bol. Soc. Esp. Mat. Apl. SeMA}, (49):33--44, 2009.

\bibitem{War}
M.~Warma.
\newblock The fractional relative capacity and the fractional {L}aplacian with
  {N}eumann and {R}obin boundary conditions on open sets.
\newblock {\em Potential Anal.}, 42(2):499--547, 2015.

\bibitem{War-In}
M.~Warma.
\newblock The fractional {N}eumann and {R}obin type boundary conditions for the
  regional fractional {$p$}-{L}aplacian.
\newblock {\em NoDEA Nonlinear Differential Equations Appl.}, 23(1):Art. 1, 46,
  2016.

\bibitem{War-AA}
M.~Warma.
\newblock On the approximate controllability from the boundary for fractional
  wave equations.
\newblock {\em Appl. Anal.}, 96(13):2291--2315, 2017.

\bibitem{War-ACE}
M.~Warma.
\newblock Approximate controllabilty from the exterior of space-time fractional
  diffusion equations with the fractional {L}aplacian.
\newblock {\em arXiv:{1802.08028}}, 2018.

\bibitem{WZ}
M.~Warma and S.~Zamorano.
\newblock Analysis of the controllability from the exterior of strong damping
  nonlocal wave equations.
\newblock {\em arXiv preprint arXiv:1810.08060}, 2018.

\bibitem{ZL}
P.~Zhuang and F.~Liu.
\newblock Implicit difference approximation for the time fractional diffusion
  equation.
\newblock {\em J. Appl. Math. Comput.}, 22(3):87--99, 2006.

\bibitem{Zua1}
E.~Zuazua.
\newblock Controllability of partial differential equations.
\newblock {\em 3\`eme cycle. Castro Urdiales, Espagne}, 2006.

\end{thebibliography}

\end{document}